\documentclass[reqno,a4paper,12pt]{amsart}
 
\usepackage{latexsym, amsmath, amssymb, amsthm, a4}
\usepackage{graphicx}
\usepackage{subcaption}

\usepackage{amsfonts, bbold, bbm, marvosym}

\usepackage{mathrsfs}

\usepackage[mathscr]{eucal}
\font\sc=rsfs10 at 12pt

\newcommand{\g}{\gamma}
\newcommand{\G}{\Gamma}
\newcommand{\de}{\delta}

\newcommand{\e}{\epsilon}

\newcommand{\y}{\eta}

\newcommand{\vt}{\vartheta}
\newcommand{\io}{\iota}
\newcommand{\ka}{\kappa}
\newcommand{\vk}{\varkappa}
\newcommand{\la}{\lambda}

\newcommand{\m}{\mu}
\newcommand{\n}{\nu}
\newcommand{\x}{\xi}

\newcommand{\ro}{\rho}

\newcommand{\s}{\sigma}
\newcommand{\Si}{\Sigma}

\newcommand{\F}{\Phi}

\newcommand{\om}{\omega}
\newcommand{\Om}{\Omega}

\newcommand{\R}{{\mathbb R}}

\newcommand{\ab}{{\mathbf a}}

\newcommand{\hb}{{\mathbf h}}

\newcommand{\jb}{{\mathbf j}}
\newcommand{\kb}{{\mathbf k}}

\newcommand{\nb}{{\mathbf n}}

\newcommand{\rb}{{\mathbf r}}
\newcommand{\sbb}{{\mathbf s}}

\newcommand{\vb}{{\mathbf v}}
\newcommand{\wb}{{\mathbf w}}

\newcommand{\Ab}{{\mathbf A}}

\newcommand{\Bb}{{\mathbf B}}

\newcommand{\Db}{{\mathbf D}}

\newcommand{\Fb}{{\mathbf F}}

\newcommand{\Jb}{{\mathbf J}}
\newcommand{\Kb}{{\mathbf K}}

\newcommand{\Nb}{{\mathbf N}}

\newcommand{\Qb}{{\mathbf Q}}
\newcommand{\Rb}{{\mathbf R}}
\newcommand{\Sbb}{{\mathbf S}}
\newcommand{\Tb}{{\mathbf T}}
\newcommand{\Ub}{{\mathbf U}}
\newcommand{\Vb}{{\mathbf V}}
\newcommand{\Wb}{{\mathbf W}}

\newcommand{\Zb}{{\mathbf Z}}

\newcommand{\aF}{\mathfrak a}

\newcommand{\AF}{\mathfrak A}

\newcommand{\dF}{\mathfrak d}

\newcommand{\GF}{\mathfrak G}
\newcommand{\HF}{\mathfrak H}

\newcommand{\mF}{\mathfrak m}

\newcommand{\Ac}{{\mathcal A}}

\newcommand{\Dc}{{\mathcal D}}

\newcommand{\Hc}{{\mathcal H}}

\newcommand{\Kc}{{\mathcal K}}
\newcommand{\Lc}{{\mathcal L}}

\newcommand{\Nc}{{\mathcal N}}
\newcommand{\Oc}{{\mathcal O}}

\newcommand{\Sc}{{\mathcal S}}

\newcommand{\Wc}{{\mathcal W}} 

\newcommand{\Ds}{\sc\mbox{D}\hspace{1.0pt}}
\newcommand{\Ls}{\sc\mbox{L}\hspace{1.0pt}}

\newcommand{\Ns}{\sc\mbox{N}\hspace{1.0pt}}

\newcommand{\dist}{{\rm dist}\,}

\newcommand{\sgn}{\hbox{{\rm sign}}\,}
\newcommand{\supp}{\hbox{{\rm supp}}\,}

\newcommand{\diam}{\operatorname{diam\,}}
\newcommand{\Ker}{\hbox{{\rm Ker}}\,}

\newcommand{\codim}{\operatorname{codim\,}}

\newcommand{\loc}{\operatorname{loc}}

\newcommand{\nub}{\pmb{\nu}}

\newenvironment{dedication}{\itshape\center}{\par\medskip}

\newtheorem{thm}{Theorem}[section]
\newtheorem{cor}[thm]{Corollary}
\newtheorem{lem}[thm]{Lemma}

\newtheorem{proposition}[thm]{Proposition}

\theoremstyle{definition}
\newtheorem{defin}[thm]{Definition}
\newtheorem{exa}[thm]{Example}

\newtheorem*{rem}{Remark}

\numberwithin{equation}{section}

\begin{document}

\title[Resolvent perturbations]{Spectral properties of the resolvent difference for singularly perturbed operators}

\author{Grigori Rozenblum}
\address{Chalmers Univ. of Technology Gothenburg Sweden and St. Petersburg University, 7/9 Universitetskaya nab., St. Petersburg, 199034 Russia}
\email{$\mathrm{grigori@chalmers.se}$}

\begin{abstract}We obtain order sharp spectral estimates for the difference of resolvents of singularly perturbed elliptic operators $\Ab+\Vb_1$ and $\Ab+\Vb_2$ in a domain $\Omega\subseteq \R^\Nb$ with perturbations $\Vb_1, \Vb_2$ generated  by  $V_1\mu,V_2\mu,$ where $\mu$ is a measure singular with respect to the Lebesgue measure and satisfying two-sided or one-sided conditions of  Ahlfors type, while $V_1,V_2$ are weight functions subject to some integral conditions. As an important special case, spectral estimates for the difference of resolvents of two  Robin  realizations of the operator $\Ab$ with different weight functions are obtained. For the case when the support of the measure is a compact Lipschitz hypersurface in $\Om$ or, more generally, a rectifiable set of Hau{\ss}dorff dimension $d=\Nb-1$, the Weyl type asymptotics for eigenvalues is justified.
\end{abstract}
\thanks{The work of G.R. was performed at the Saint Petersburg Leonhard Euler International Mathematical Institute and supported by the Ministry of Science and Higher Education (Agreement No. 075--15--2022--287).}
\maketitle

\begin{dedication}
Dedicated to Gerd Grubb on her Jubilee.
\end{dedication}
\section{Introduction}\subsection{Resolvent perturbations} Let $\Ab$ be a self-adjoint second order elliptic differential operator in an open (possibly, unbounded) connected set $\Om\subseteq \R^{\Nb}$.  Suppose that $\Vb_1$ and $\Vb_2$ are some  perturbations of $\Ab$ (one of them may be the zero one), so, at least symbolically, \begin{equation}\label{perturbations}
  \Ab_1=\Ab+\Vb_1, \, \Ab_2=\Ab+\Vb_2.
\end{equation}
Generally, the perturbations $\Vb_1,\Vb_2$  can be rather singular, and the expressions for $\Ab_1,\Ab_2$ in \eqref{perturbations} can be understood in a quite generalized and often highly formal way. In particular, such a vague setting may  embrace the change in the boundary conditions for an elliptic operator. The question about the spectral properties of the difference of resolvents, 
 \begin{equation}\label{difference}
  \Rb^{(1,2)}=(\Ab_1+t)^{-1}-(\Ab_2+t)^{-1},
\end{equation}
where  the real number $t$
is chosen in such way that both operators in \eqref{difference} are invertible, is one of the classical ones in Spectral Theory, with long history. In particular, an efficient method was elaborated and fundamental results were obtained in \cite{Birman.61},  \cite{Birman.62}, where one can find referrals to earlier papers and a motivation for studying this kind of problems, the most important being the needs of the scattering theory. Further on, quite general abstract operator perturbation  methods have been developed, see, in particular,  \cite{AGW}, \cite{BeBook}, \cite{BLLLP}, \cite{BLL}, \cite{GeMit}, \cite{GeMi2}, \cite{GeMi3}, \cite{Grubb.Krein}, \cite{Post}, \cite{Malamud} and many other sources. Being applied to concrete differential operators, these abstract perturbation schemes have produced a number of impressive results on the spectral properties, semigroup properties, regularity etc.

One of important directions of study is the investigation of the spectrum of the \emph{difference of resolvents} of operators generated by the same elliptic differential expression with different boundary conditions. For the difference of resolvents of the Dirichlet and Neumann operators, eigenvalue estimates we obtained in the already cited paper  \cite{Birman.62} by M.Birman in 1962, while later, in \cite{BSSib}, general methods were developed producing eigenvalue \emph{asymptotics} for many problems, including the above one.

Even this, rather particular, topic attracted great attention in the latest two or even three decennia, with scores of papers published. We mention here the papers \cite{AGW}, \cite{Grubb.Krein}, \cite{Grubb.nonsmooth}, \cite{BLL}, \cite{BeGuLo}, and a vast amount of  references therein. For the study of these problems, the extensive machinery of abstract Krein-von-Neumann perturbation formulas, quasi-boundary triples,  and the generalized H.Weyl $M$-function have been elaborated. This approach, in various modifications, was used to the study of many kinds of spectral problems, a systematic description and an impressive bibliography can be  found in \cite{BeBook}.

One of  special singular problems under consideration has been the study of the spectral properties of the difference of resolvents of self-adjoint operators $\Ab_\io,$ $\io=1,2,$ in $\Om\subseteq\R^\Nb$, corresponding to the Laplacian (or a general symmetric  second order elliptic operator $\Lc$) with weighted Robin boundary conditions containing two different weight functions,
\begin{equation*}
  \partial_{\n(\Lc)} u(X)=V_\io(X)u(X), \, \io=1,2, X\in\G=\partial\Om,
\end{equation*}
where $\partial_{\n(\Lc)} u$ is the (co)normal derivative at the boundary, associated with $\Lc$. For the singular values $s_j  $ of this difference, in \cite{Birman.62}, the estimate $s_j=O(j^{-\frac{2}{\Nb-1}})$ was obtained, it was improved to $s_j=O(j^{-\frac{3}{\Nb-1}})$ in \cite{BLL}. Further on, in \cite{GrRobin}, the latter  estimate was elevated to the asymptotic formula
\begin{equation}\label{As.init}
  s_j\sim (\Ac j)^{-\frac{3}{\Nb-1}}, \, \Ac=C\int_{\Si}|V_1(X)-V_2(X)|^{\frac{\Nb-1}{3}}\omega(X)\s(dX),
\end{equation}
where the density $\omega(X)$ is determined (in some complicated way, see Theorem 3.5 in \cite{GrRobin}) by the coefficients of the operator $\Lc$ at the boundary. We do not need the particular expression for this weight function.

Formula \eqref{As.init} was established in \cite{GrRobin} under rather restrictive regularity conditions for the functions $V_\io$, namely, that they and their difference should be piecewise H{\"older continuous, with allowed discontinuities only along some smooth hypersurfaces  in $\G$, common for  $V_1,V_2$ (in fact, a little less restrictive regularity  is required, expressed in terms of spaces of Bessel potentials.) These restrictions were caused by the pseudodifferential technique used in   \cite{GrRobin}; further on, the regularity conditions were somewhat relaxed in \cite{Grubb.nonsmooth}.  Important results were also obtained for various singular perturbations  in \cite{BLLLP},  \cite{GeMit}, \cite{Malamud}, and many more.

Another perturbation problem, also subject to an extensive study, is the analysis of spectral propertied of the difference of resolvents of 'Schr{\"o}dinger-like operators,'
\begin{equation}\label{Pert.Schr}
  (\Ab+V_1(X)+t)^{-1}-(\Ab+V_2(X)+t)^{-1},
\end{equation}
where $V_1,V_2$ are\emph{ functions }defined in $\Om.$ Such problems were being studied systematically, starting, 
probably, with \cite{Beals} and \cite{Yafaev}. For functions $V_{\io}$ with compact support and some summability conditions, eigenvalue estimates and eigenvalue asymptotics were known for the operator \eqref{Pert.Schr} since long ago.  More generally, operators of the type $-\Delta+V$, with the \emph{ singular} perturbation $V$ being the $\de$-distribution supported on a hyper-surface, with some weight, were being considered, see, especially, \cite{ExnerLeaky}, \cite{Exner20} and references therein; such mathematical models were treated as 'leaky quantum graphs'. More recently, the   studies were concerned with the spectrum of resolvent differences for such singularly perturbed  operators,   see, e.g.,  \cite{BeGuLo}, \cite{BEHL}, \cite{Exner}. The coefficients in front of the distributional potential were supposed to be smooth functions on the hypersurface.
More generally,  one can consider singular perturbations problems for  Schr{\"o}\-dinger-like operators, as above, but with the roles of potentials  played by singular measures supported on some  sets of zero Lebesgue measure, possibly, with fractional dimension, sometimes called 'transmission across a fractal set,' see, e.g., \cite{AlKosh}, \cite{AlKosh2}, \cite{Br95}, \cite{Capitanelli}, \cite{Goldberg}, \cite{Johnsson}, \cite{Kumagai}, \cite{Lancia}, \cite{Lancia2}, etc.
\subsection{Singular perturbations}
In the present paper, we show that the study of spectral properties of the  problems discussed above can be performed in a uniform way,  as  versions of one and the same construction.
Namely, for a Radon measure $\m$ supported in $\overline{\Om}$, possibly, singular with respect to the Lebesgue measure, and a $\m$-measurable function $V$ on the support of $\m$ and subject to some summability conditions, the operator $\Ab+V\m$ in $L_2(\Om)$ can be effectively  defined by means of quadratic forms, and the eigenvalues of the difference of its resolvent and the resolvent of $\Ab$ (as well as of the difference of powers of resolvents) can be estimated.
 The order of such estimates is determined by the class of Ahlfors regularity of the measure $\m$ (roughly speaking,  the local Hau{\ss}dorff dimension of the support of $\m$), while the coefficient in the estimates depends on a certain integral norm of the 'weight' function $V$ with respect to the measure $\m$. In a similar way, the spectrum of the difference of  (powers of) resolvents of operators with two different weight functions $V_1, V_2$ can be studied, producing eigenvalue estimates and, in certain cases, asymptotics, generalizing \eqref{As.init}.

In this setting, the Robin problem corresponds to the measure $\m$ supported on some set in the boundary of $\Om,$ while the latter type of problems, which we call 'the Schr{\"o}dinger type'  ones, corresponds to the measure $\m$ supported inside $\Om.$ It is possible, of course, to consider a combination of these perturbations.

Eigenvalue estimates obtained in the present paper involve integral characteristics of the weight functions. This circumstance  provides us  quite a freedom in extending the existing results on the \emph{eigenvalue asymptotics} to considerably less regular weight functions $V_{\io}$ for the difference of Robin resolvents, compared, e.g.,  with the results in \cite{GrRobin}. Note that in addition to the asymptotics of singular numbers, studied previously, we are able to separate positive and negative parts of $V_1-V_2$, thus finding the asymptotics of negative and positive eigenvalues separately.   Similarly, for   problems with singular potentials supported inside $\Om$, the eigenvalue asymptotics is found for the measure $\m$ 
being the Hau{\ss}dorff measure on a Lipschitz hypersurface in $\Omega$ or on some unions of such surfaces (uniformly rectifiable sets).  When the set $\Si=\supp\mu$ is Ahlfors regular of a fractional dimension, asymptotic formulas for eigenvalues are, probably, not accessible in a general case, however,  our estimates are order sharp: the general upper estimates are accompanied by lower estimates of the same order.    

There is a huge literature on the operator perturbation theory, with extensive machinery, involving boundary triples, generalized boundary pairs, Weyl $M$-functions etc.; this activity has produced quite a lot of impressive results. In our case, where we deal with perturbation of operators defined by means of quadratic forms, using the approach originating in  \cite{Birman.61} and developed further in \cite{Br95} and  in \cite{Post}; the latter is  the closest to ours, being much more general. We, however, do not need this approach in its full generality; our perturbation scheme is rather elementary and does not use the above general constructions, giving at the same time convenient explicit formulas.

In dealing with spectral estimates and asymptotics for  singular measures, we follow the approach developed recently  in the papers  \cite{KarShar}, \cite{RSh}, \cite{GRInt}, \cite{RTsing}, \cite{GRLT} concerning spectral properties of Birman-Schwinger type operators with singular weights and use some important results of these papers, especially, \cite{RTsing}. When needed, we reproduce some formulations  from these papers; we explain the strategy of proving asymptotic formulas as well.

In this paper, we restrict ourselves to perturbations and boundary problems for  second order uniformly elliptic operator $\Lc$ with smooth coefficients. The regularity of the boundary, if it is present, depends on the particular problem under consideration. More general results involving  elliptic operators with less regular coefficients and boundaries (like in \cite{GRLip}), as well as higher order operators,  would require  essentially  more complicated considerations which would dim simple ideas in our approach to spectral problems. We hope to  return to these general topics on some other occasion.

There are numerous papers devoted to various aspects of spectral theory of singular perturbations. Our initial attempts to include a complete bibliography led to the reference list filling almost one third of the paper. Therefore, we decided to make the reference list  (more or less) reasonably   short, so that the interested Reader may find an additional information using reference lists of the papers cited here and the citation lists in the entries of these papers in Math.Sci.Net.

\section{The abstract perturbation formula and related spectral properties}\label{Sect.Abstr.Pert}
We consider perturbations of a given operator by quadratic forms. The formulas for the resolvent of the perturbed operator are somewhat similar to the ones established in \cite{Br95} and, later, in \cite{Post}, however this particular form  enables us to apply previously obtained results concerning spectral estimates.
\subsection{Form perturbations}\label{Subsect.Peturbations}
Let $\Ab>0$ be a self-adjoint operator  in the Hilbert space $\HF.$ We suppose that $\Ab\ge c_0>0$ 
is defined by the positive quadratic form $\ab[u]=(\Bb u,\Bb u)_{\HF}$ with domain $\HF^1=\Dc(\Ab^{\frac12}).$ Here $\Bb$ is a bounded  operator acting from the Hilbert space $\HF^1$ to $\HF$. Of course, there is a considerable freedom in choosing  the   operator $\Bb$  
 generating  one and the same operator $\Ab.$ The choice, always possible, is $\Bb=\Ab^{\frac12}$. 

The perturbation is supposed to be defined by means of the quadratic form $\vb[u]$ with domain $\dF[\vb]\supset\Dc(\Ab^{\frac12})\equiv\HF^1$ and bounded there,
\begin{equation}\label{ab.bounded}
  |\vb[u]|\le C \ab[u], u\in\HF^1.
\end{equation}
Therefore, the quadratic form $\vb[u]$ defines a bounded self-adjoint operator $\Vb$ in $\HF^1$ satisfying
\begin{equation*}
  (\Vb u,u)_{\HF^1}=\vb[u], u\in\HF^1.
\end{equation*}

Under these conditions, we define the  perturbed operator $\Ab_{\Vb}$ in $\HF$  by means of the quadratic form
\begin{equation}\label{perturbed form0}
  \ab_\Vb [u] = \ab[u]+\vb[u], u\in \HF^1,
\end{equation}
In this paper, we suppose that  the quadratic form \eqref{perturbed form0} is positive in $\HF$.
Then, by the KLMN theorem, this quadratic form is closed in $\HF$ and, in fact, defines a lower semi-bounded self-adjoint operator $\Ab_{\Vb}$.
Having the quadratic form definition of the operator $\Ab_{\Vb}$, we are going now to describe conveniently the operator itself.

Set $u=\Ab^{-\frac12}v,$ $v\in \HF$, in \eqref{perturbed form0}.
We obtain
\begin{equation*}
  \vb[u]=\vb[\Ab^{-\frac12}v].
\end{equation*}

We suppose that the perturbing quadratic form  has the following structure. Let $\GF$ be a Hilbert space and $\g$ be a (boundary) operator,
$\g:\HF^1\to\GF$ such that  
\begin{equation*}
  \|\g u\|^2_{\GF}\le C \|u\|^2_{\HF^1}, \, u\in\HF^1. 
\end{equation*}
Further on, let $\Fb$ be an operator in the space $\GF$ such that its domain $\Dc(\Fb)$ contains $\g \HF^1$ and 
\begin{equation*}
  \|\Fb\g u\|_{\GF}\le C\|u\|_{\HF^1}
\end{equation*}
(note that we do not expect that $\Fb$ is a bounded operator in $\GF$.)
Finally, let $\Ub$ be a bounded self-adjoint operator  in $\GF$. 

Under the above conditions, the quadratic form $\vb$ is supposed to be defined by 
\begin{equation}\label{qform}
  \vb[u]=(\Ub \Fb \g u, \Fb\g u)_{\GF}, u\in \HF^1.
\end{equation}

The expression \eqref{qform} is defined, at least, on $\HF^1$ and it is bounded with respect to the $\HF^1$ norm, see \eqref{ab.bounded}. Since $u=\Ab^{-\frac12}v$, $v\in \HF$,  \eqref{qform} transforms to

\begin{equation}\label{qform2}
  \vb[\Ab^{-\frac12}v]=(\Ub \Fb\g \Ab^{-\frac12}v,\Fb\g \Ab^{-\frac12}v )_{\GF}.
\end{equation}
By the conditions imposed above,   $\Fb\g \Ab^{-\frac12}$ is a bounded operator from $\HF$ to $\GF.$ Therefore, the expression in \eqref{qform2} can be represented as

\begin{equation}\label{qform3}
  \vb[\Ab^{-\frac12}v]=((\Fb\g \Ab^{-\frac12})^* \Ub (\Fb\g \Ab^{-\frac12})v,v)_{\HF}.
\end{equation}
Here, $(\Fb\g \Ab^{-\frac12})^*$ acts from $\GF$ to $\HF$; it is the adjoint operator  to   $\Fb\g \Ab^{-\frac12}$ considered as acting from $\HF$ to $\GF.$ Since the operator $\Ub$ is self-adjoint and bounded, the operator $(\Fb\g \Ab^{-\frac12})^* \Ub (\Fb\g \Ab^{-\frac12})$ is a self-adjoint bounded operator in $\HF$.

We return to $v=\Ab^{\frac12}u$. The expression in \eqref{qform3} transforms to
\begin{equation*}
\vb[u]=((\Fb\g\Ab^{-\frac12})^*\Ub(\Fb\g\Ab^{-\frac12})\Ab^{\frac12}u, \Ab^{\frac12}u)_{\HF}, \, u\in\HF^1.
\end{equation*}
Now we consider the operator  $\Ab_\Vb$ defined in $\HF$ by means of the quadratic form \eqref{perturbed form0}, 
with  $\vb$ as in \eqref{qform}.
We suppose further on that the form $\ab_\Vb[u]$ is positive in $\HF^1$, namely, 
 \begin{equation}\label{qform6}
   \ab_\vb[u]\ge C\ab[u], \, u\in\HF^1,
 \end{equation}
 with some $C>0$. Under this condition, the operator $1+(\Fb\g \Ab^{-\frac12})^* \Ub (\Fb\g \Ab^{-\frac12})$ is positive, and therefore, we can write, for $u\in\HF^1,$
 \begin{gather}\label{qform7}
 \ab_\Vb[u]=(\Ab^{\frac12}u,\Ab^{\frac12}u )_{\HF}+((\Fb\g \Ab^{-\frac12})^* \Ub (\Fb\g \Ab^{-\frac12})\Ab^{\frac12}u, \Ab^{\frac12}u)_{\HF}=\\\nonumber
 ((1+(\Fb\g \Ab^{-\frac12})^* \Ub (\Fb\g \Ab^{-\frac12}))\Ab^{\frac12}u, \Ab^{\frac12}u)_{\HF}=\\\nonumber
 ((1+(\Fb\g \Ab^{-\frac12})^* \Ub (\Fb\g \Ab^{-\frac12})^{\frac12})\Ab^{\frac12}u, (1+(\Fb\g \Ab^{-\frac12})^* \Ub (\Fb\g \Ab^{-\frac12})^{\frac12})\Ab^{\frac12}u)_{\HF}.
 \end{gather}
The last expression in \eqref{qform7} still defines the quadratic form of the self-adjoint operator $\Ab_\Vb$, and therefore we have the representation for this operator  as the composition of an operator and its adjoint (such composition is always self-adjoint):
  \begin{equation}\label{Operator}
   \Ab_\Vb =\left[(1+(\Fb\g \Ab^{-\frac12})^* \Ub (\Fb\g \Ab^{-\frac12}))^{\frac12}\Ab^{\frac12}\right]^*\left[(1+(\Fb\g \Ab^{-\frac12})^* \Ub (\Fb\g \Ab^{-\frac12}))^{\frac12}\Ab^{\frac12}\right].
 \end{equation}
Note again that we do not (and need not to) describe explicitly the domain of this operator.
 The operator $\Ab_\Vb$ is invertible, since  in \eqref{Operator} it is represented as a composition of two invertible operators. Therefore, for the  (bounded)
 inverse operator to $\Ab_\Vb$, we obtain the representation
 \begin{gather}\label{inverse}
   (\Ab_\Vb)^{-1}=[(1+(\Fb\g \Ab^{-\frac12})^* \Ub (\Fb\g \Ab^{-\frac12}))^{\frac12}\Ab^{\frac12}]^{-1}\times \\\nonumber [((1+(\Fb\g \Ab^{-\frac12})^* \Ub (\Fb\g \Ab^{-\frac12}))^{\frac12}\Ab^{\frac12})^*]^{-1}=\\\nonumber
   \Ab^{-\frac12}(1+(\Fb\g \Ab^{-\frac12})^* \Ub (\Fb\g \Ab^{-\frac12}))^{-\frac12}\times\\\nonumber
   (1+(\Fb\g \Ab^{-\frac12})^* \Ub (\Fb\g \Ab^{-\frac12}))^{-\frac12}\Ab^{-\frac12}=\\\nonumber
   \Ab^{-\frac12}(1+(\Fb\g \Ab^{-\frac12})^* \Ub (\Fb\g \Ab^{-\frac12}))^{-1}\Ab^{-\frac12}\\\nonumber \equiv \Ab^{-\frac12}(1+\Tb)^{-1}\Ab^{-\frac12}, \, \,\Tb\equiv \Tb_\Vb=(\Fb\g \Ab^{-\frac12})^* \Ub (\Fb\g \Ab^{-\frac12}).
 \end{gather}
 
 We arrive here at the resolvent perturbation formula, which creates the base of our further considerations.
 
 \begin{proposition}Let $\Ab$ be a positive self-adjoint operator and let $\Vb$ be a perturbation as above. Then
 \begin{gather}\label{ResolventDiff}
 \Rb_{\Vb}\equiv  \Ab^{-1}-(\Ab_\Vb)^{-1}=\Ab^{-\frac12}\left[1- (1+(\Fb\g \Ab^{-\frac12})^* \Ub (\Fb\g \Ab^{-\frac12}))^{-1}\right]\Ab^{-\frac12}\equiv\\\nonumber \Ab^{-\frac12}\left[1-(1+\Tb_{\Vb})^{-1}\right]\Ab^{-\frac12}.
 \end{gather}
 \end{proposition}
 
 Operators of the form $\Tb$ and their analogies have been the base of the study of discrete spectrum of Schr{\"o}dinger-type operators since 60-s. It is reasonable to call them 'Birman-Schwinger' operators.
 
The representation \eqref{ResolventDiff} of the resolvent difference  has been derived under the condition that the operator
 \begin{equation*}
   1+\Tb_{\Vb}\equiv\left[1+(\Fb\g \Ab^{-\frac12})^* \Ub (\Fb\g \Ab^{-\frac12})\right]
 \end{equation*}
 is positive definite. If we replace the operator $\Ab$ by $\Ab-t,$ $t<0$,  \eqref{qform6} still holds, and \eqref{inverse} gives  the representation for the inverse of $\Ab_\Vb-t,$ this means, for the resolvent of the operator $\Ab_\Vb$ at the point $-t$. If there is a family of quadratic forms $\vb_\e$ such that the corresponding operators $\Tb_{\Vb_\e}$ converge in the norm sense to $\Tb_{\Vb}$ as $\e\to 0$, then, as soon as $\|\Tb_{\Vb_\e}-\Tb_{\Vb}\|$ is small enough, 
 the operator $1+\Tb_{\Vb_{\e}}$ is positive definite as well and the resolvent perturbation formula  \eqref{ResolventDiff} is valid.

 \subsection{Asymptotic spectral characteristics}\label{Subsect.spectral}
  We are going to study asymptotical spectral properties of compact operators in a Hilbert space. Here we collect some basic notations and facts, mostly well known. 
Let $\Kb$ be a compact self-adjoint  operator in a Hilbert space $\HF$. By $n_{\pm}(\la,\Kb)$ we denote the distribution function of positive (resp., negative) eigenvalues  $\pm\la^{\pm}_{j}(\Kb)$ of the operator $\Kb$:
\begin{equation*}
  n_{\pm}(\la,\Kb)=\#\{j: \la_j^{\pm}(\Kb)>\la\}.
\end{equation*}
For a given exponent $\theta>0$, the asymptotic characteristics of the spectrum of $\Kb$ are defined
as 
\begin{equation}\label{limsup}
  \nb^{\sup}_{\pm}(\Kb,\theta)=\limsup_{\la\to +0}n_{\pm}(\la,\Kb)\la^{\theta}, \, \nb^{\inf}_{\pm}(\Kb,\theta)=\liminf_{\la\to +0}n_{\pm}(\la,\Kb)\la^{\theta}
\end{equation}
and 
\begin{equation}\label{limlim}
  \nb_{\pm}(\Kb,\theta)=\lim_{\la\to +0}n_{\pm}(\la,\Kb)\la^{\theta},
\end{equation}
provided these limits exist. 

Similarly, for an \emph{arbitrary} compact operator $\Kb$, the distribution function for its singular numbers $s_j(\Kb)$  and their asymptotic characteristics are defined and denoted, as above, but with the symbol $\pm$ deleted.

It follows immediately that for a bounded operator $\Jb$,
  \begin{equation*}
    \nb(\la, \Jb\Kb)\le \nb(\la, \|\Jb\|\Kb)=\nb(\|\Jb\|^{-1}\la,\Kb),
  \end{equation*}
  therefore,
\begin{equation*}
  \nb^{\sup}(\Jb\Kb,\theta)\le\|\Jb\|^{\theta}\nb^{\sup}(\Kb,\theta).
\end{equation*}


The $\limsup$ form of the multiplicative Ky Fan  inequality 
\begin{equation}\label{KyFan.S}
  n(\la_1\la_2,\Kb_1\Kb_2)\le n(\la_1,\Kb_1)+n(\la_2,\Kb_2)
\end{equation}
implies in these terms 
\begin{equation}\label{kFmult8}
  \nb^{\sup}(\Kb_1\Kb_2, \theta)\le 2\nb^{\sup}(\Kb_1,\theta_1)^{\frac{\theta}{\theta_1}}\nb^{\sup}(\Kb_2,\theta_2)^{\frac{\theta}{\theta_2}}.
\end{equation}
Similar inequalities are valid for the product of several operators,  $\Kb=\prod_{j=1}^m \Kb_j,$ $\theta^{-1}=\sum\theta_j^{-1}$:
\begin{equation}\label{mult.m}
   \nb^{\sup}(\Kb, \theta)\le m\prod \left(\nb^{\sup}(\Kb_j,\theta_j)^{\frac{\theta}{\theta_j}} \right).
\end{equation}

Another simple consequence  of the multiplicative Ky Fan inequality is the following:
\begin{lem}\label{lem.comp}
  If $\nb^{\sup}(\Kb_j,\theta_j)<\infty,$ for $j=1,3$ and $\Kb_2$ is a compact operator, then 
 \begin{equation*}
   \nb^{\sup}(\Kb_1\Kb_2\Kb_3, \theta)=0, \, \theta^{-1}=\theta_1^{-1}+\theta_2^{-1}.
 \end{equation*}
\end{lem}
 
\subsection{Ratios of quadratic forms}\label{subsect.ratios}
We will systematically use operators defined by the ratio of quadratic forms, see, e.g.,  \cite{BS}.

If $\kb$ is a bounded Hermitian  quadratic form in the Hilbert space $\HF$, its eigenvalues (eigenvalues of the operator $\Kb$ defined by this quadratic form) are determined in variational way, in particular,
\begin{equation*}
  n_{\pm}(\la,\Kb)=\min\{\codim\Ls\subset\HF: \frac{\pm \kb[u]}{\hb[u]}>\la, u\in\Ls\setminus\{0\}\},
\end{equation*}
where $\hb[u]:=\|u\|^2_{\HF}.$
It is convenient to denote this distribution function by $ n_{\pm}(\la,\frac{\kb}{\hb})$. The notation of asymptotic characteristics \eqref{limsup}, \eqref{limlim} is modified accordingly.

When we handle several different Hilbert spaces (related, in  concrete examples later, with different domains in $\R^\Nb$), the following statement is useful, see Lemma 1.2 in \cite{BS}, which we reproduce here, in our notations.
\begin{lem}\label{BSLem.extension}
  Let $\HF_1, \HF_2$ be two Hilbert spaces, with squared norm $\hb_1,\hb_2$, and  let $\Kb_1,\Kb_2 $ be self-adjoint compact operators in these spaces. We denote by $\kb_1, \kb_2$ the quadratic forms of these operators, $\kb_\io[u]=(\Kb_\io u,u)_{\HF_\io}, \, \io=1,2. $ Let $\Jb: \Hc_1\to\Hc_2$ be a bounded operator, moreover, $\kb_1[u]=0$ for all $u\in \Ker(\Jb).$     Suppose that for all $u\in\HF_1$,
  \begin{equation*}
    \pm\frac{\kb_1[u]}{\hb_1[u]}\le \pm\frac{\kb_2[\Jb u]}{\hb_2[\Jb u]}
  \end{equation*}
  for all $u\in \HF_1,$ for which $\pm\kb_1[u]>0.$
  Then 
  \begin{equation*}
  n_{\pm}(\la,\Kb_1)\le n_{\pm}(\la,\Kb_2),
  \end{equation*}
  for all $\la>0,$ with the corresponding inequalities for asymptotical spectral characteristics.
\end{lem}

This statement will be used on two concrete occasions.
\begin{cor}\label{cor.restr} Let $\Om_2\subset\Om_1\subseteq\R^\Nb$.\begin{enumerate}
  \item Let $\HF^{(l)}$ be some space of functions in $H^{l}_{\loc}(\Om_1)$. Suppose that  the operator  $\Jb$  of \emph{restriction} of functions in $\HF^{(l)}$  to $H^l(\Om_2)$ is bounded. We set   $\hb_2[u]=\|u\|_{H^l(\Om_2)}^2$, $\hb_1[u]= \|u\|^2_{\HF^{(l)}(\Om_1)}$ and suppose that $\kb_2[\Jb u]=\kb_1[u].$ Then   
  \begin{equation}\label{KF.est}
    n_{\pm}(\la, \Kb_1)\le n_{\pm}(\la, \|\Jb\|\Kb_2)=n_{\pm}(\|\Jb\|^{-1}\la, \Kb_2).
  \end{equation}
\item Let $\Om_1\subset\Om_2\subseteq\R^\Nb$ and $\Jb$ be the bounded operator of  \emph{extension} of functions in $H^l(\Om_1)$ to $H^l(\Om_2)$; suppose that $\kb_2[\Jb u]=\kb_1[u].$ Then \eqref{KF.est} holds. 
\end{enumerate}
\end{cor}
\subsection{Asymptotic perturbations}\label{Sub.As.Perturb}
The major tool in the study of \emph{spectral asymptotics} for problems involving singularities of various kinds is the asymptotic perturbation Lemma by M.Birman and M.Solomyak, see, e.g., \cite{BS}, Lemma 1.5.
We reproduce it here, in our present notations.
\begin{lem}\label{BSlemma}
  Let $\Kb$ be a compact self-adjoint operator, and there exists a sequence $\Kb_\e$, $\e\to 0,$ such that
   \begin{description}\item[(i)] for some $\theta>0,$ the limits $\nb_{\pm}(\Kb_\e,\theta)=\mF_{\pm,\e}$ exist and are finite;
   \item[(ii)] $\lim_{\e\to 0}\nb^{\sup}(\Kb-\Kb_\e, \theta)= 0.$
   \end{description} Then the limits $\mF_{\pm}=\lim_{\e\to 0} \mF_{\pm,\e}$ exist and
  $\nb_{\pm}(\Kb,\theta)=\mF_{\pm}.$ \\
  If the condition \textbf{\emph{(i)}} above is satisfied for only one of the signs $\pm,$ the conclusion of the Lemma holds for this sign.
\end{lem}
Note that often, say, in \cite{GrRobin}, Lemma 3.1, the condition \textbf{(ii)} above is stated in a stronger form, requiring not the asymptotic but a uniform in $\la$ estimate for singular numbers of the difference $\Kb-\Kb_\e$. In our study, it is the less restrictive condition \textbf{(ii)} that can actually be guaranteed, while uniform estimates may be  not available.
 
\subsection{Spectral estimates for form perturbations}In this subsection we reduce the task of finding eigenvalue estimates for the perturbations of the resolvent to the one for Birman-Schwinger type operators. We start with a simple transformation.
\begin{lem}\label{Lem.mainterm}Let $\Ab,$ $\Ab_{\Vb}$ be the operators constructed in Section \ref{Subsect.Peturbations},
  \begin{equation}\label{Tb0}
  \Tb= \Tb_{\Vb}= (\Fb\g \Ab^{-\frac12})^* \Ub (\Fb\g \Ab^{-\frac12})
   \end{equation}
 Then the following identity holds
\begin{gather}\label{Main term}
\Ab^{-1}-\Ab_{\Vb}^{-1}\equiv \Rb_{\Vb}=\Ab^{-\frac12}\Tb\Ab^{-\frac12}-\Ab^{-\frac12}\Tb(1+\Tb)^{-1}\Tb\Ab^{-\frac12}\\\nonumber
  \equiv\Rb_{\Vb}^{(1)}-\Rb_{\Vb}^{(2)}.
\end{gather}  
\end{lem} 
\begin{proof}The relation \eqref{Main term} follows immediately from the obvious identity 
  $1-(1+\Tb)^{-1}=\Tb-\Tb(1+\Tb)^{-1}\Tb.$
\end{proof}
Next we transform the terms $\Rb_{\Vb}^{(1)}$ and $\Rb_{\Vb}^{(2)}$ in \eqref{Main term}.

\begin{lem}\label{Lem. Main term trans} The terms in \eqref{Main term} are equal to
  \begin{equation}\label{R1}
    \Rb_{\Vb}^{(1)}=(\Fb\g\Ab^{-1})^*\Ub(\Fb\g\Ab^{-1}),
  \end{equation}
  and
  \begin{equation}\label{R2}
    \Rb_{\Vb}^{(2)}=(\Fb\g\Ab^{-1})^*\Ub(\Fb\g\Ab^{-\frac12})(1+\Tb)^{-1}(\Fb\g\Ab^{-\frac12})^*\Ub(\Fb\g\Ab^{-1}).
  \end{equation}
\end{lem}
\begin{proof}All manipulations below are legal since we deal with bounded operators only. For \eqref{R1}, we use that $\Fb\g\Ab^{-1}=(\Fb\g\Ab^{-\frac12})\Ab^{-\frac12}$, therefore, 
\begin{equation*}
 (\Fb\g\Ab^{-1})^*= \Ab^{-\frac12}(\Fb\g\Ab^{-\frac12})^*.
\end{equation*}
The same identity takes care of the first and the last factors in \eqref{R2}.
\end{proof}
Next, we have the eigenvalue estimate for $\Rb_{\Vb}$.
\begin{lem}\label{Lem.pert.eigenv}
  Suppose that for some $\theta>0$, the following asymptotic singular numbers estimate is known:
\begin{equation*}
  \nb^{\sup}(\Fb\g\Ab^{-1}, 2\theta)= \mF<\infty,
\end{equation*}
$\|\Ub\|\le 1$, and the operator $\Fb\g\Ab^{-\frac12}$ is compact. Then
\begin{equation*}
  \nb^{\sup}(\Rb_{\Vb}, \theta)\le 2\mF.
\end{equation*}
\end{lem}
\begin{proof}  We express \eqref{R2} in the form
\begin{equation*}
  \Rb^{(2)}_{\Vb}=\Kb_1^*\Kb_2\Kb_1,\, \Kb_1=\Fb\g\Ab^{-1}, \Kb_2= \Ub(\Fb\g\Ab^{-\frac12})(1+\Tb)^{-1}(\Fb\g\Ab^{-\frac12})^*\Ub.
\end{equation*}
The operator $\Kb_2$ is compact, and we can apply Lemma \ref{lem.comp} and then  the Ky Fan inequality for the sum of compact operators. 
\end{proof}
Next we establish a relation for operators containing $\Ab^{-1}$ to a higher power. It is a direct consequence of the Ky-Fan inequality in the  limit form, see \eqref{KyFan.S}.
\begin{lem}\label{lem.higherpwer}
  Let $\Fb_1,\Fb_2$ be two operators in $\GF$ such that $\Qb_\io=\Fb_\io\g\Ab^{-l_\io/2}$ are compact operators and $\nb^{\sup}(\Qb_\io,2\theta_\io)=\mF_\io<\infty$, $\io=1,2$; let also $\Bb$ be a bounded operator in $\GF$. Then, for $\theta^{-1}=\theta_1^{-1}+\theta_2^{-1}$,
\begin{equation*}
  \nb^{\sup}(\Qb_1^*\Bb\Qb_2,\theta)\le C \|\Bb\|^{\theta} \mF_1^{\frac{\theta}{\theta_1}}\mF_2^{\frac{\theta}{\theta_2}}.
\end{equation*}
\end{lem}
More generally, iterating, if we have the product of several operators,  $\Qb_\io=\Fb_\io\g\Ab^{-l_\io/2}, \, \nb^{\sup}(\Qb_\io,2\theta_\io)=\mF_\io<\infty,$ $\io=1,\dots,2m$ and bounded operators $\Bb_{\io},$ then
\begin{equation*}
  \nb^{\sup}(\prod_{\ka=1}^{m} \Qb_{2\ka}^*\Bb_\ka\Qb_{2\ka-1},\theta)\le C \prod \|\Bb_\ka\|^{theta}\prod\mF_\io^{\frac{\theta}{\theta_\io}}.
\end{equation*}
with 
$\theta^{-1}=\sum\theta_{\io}^{-1}.$

\subsection{The difference of perturbed resolvents}\label{Subsect.diff}
Let two quadratic forms, $\vb_1[u],$ $\vb_2[u]$ be given, similarly to  \eqref{qform},
\begin{equation*}
  \vb_1[u]=(\Fb_1  u,\Ub_1 \Fb_1 u)_{\GF}, \, \vb_2[u]=(\Fb_2  u,\Ub_2 \Fb_2 u)_{\GF},\, u\in \HF^1,
\end{equation*}
both satisfying the conditions in Sect.2.1, and, additionally, the positivity condition, as in 
\eqref{qform6},

We are going to study the difference of the resolvents of the perturbed operators,
\begin{equation}\label{Res}
  \Rb^{(1,2)}=\Ab_{\Vb_1}^{-1}-\Ab_{\Vb_2}^{-1};
\end{equation}
our main interests lies in   eigenvalue estimates (and, if possible, asymptotics) for $\Rb^{(1,2)}.$ Lemma to follow provides us with a convenient representation of this difference. 
\begin{lem}\label{lem.Res.Diff} 
Suppose that for the quadratic form $\vb[u]=\vb_1[u]-\vb_2[u]$,  with $\vb_1,\vb_2$ as in \eqref{qform}, $1+\Tb_\io\equiv 1+(\Fb_\io\g\Ab^{-\frac12})^*\Ub_\io(\Fb_\io\g\Ab^{-\frac12}),$ $\io=1,2$ are invertible. Then the following identity is valid
    
    \begin{equation*}
      \Rb^{(1,2)}=\Ab^{-\frac12}(\Tb_1-\Tb_2)\Ab^{-\frac12}-\Zb_1+\Zb_2,
    \end{equation*}
    where
 \begin{equation*}
 \Zb_\io=\Ab^{-\frac12}(\Tb_\io(1+\Tb_\io)^{-1}\Tb_\io)\Ab^{-\frac12}.      
\end{equation*}
          \end{lem}
          \begin{proof}
            We express the operator \eqref{Res} as
            
             \begin{equation}\label{R12proof}\Rb^{(1,2)}=
            \left((\Ab_{\Vb_1})^{-1}-\Ab^{-1}\right) -\left((\Ab_{\Vb_2})^{-1}-\Ab^{-1}\right) \end{equation}
            and apply to both term on the right in \eqref{R12proof} the perturbation formula \eqref{Main term}. 
          \end{proof}
          From this Lemma, we derive the eigenvalue estimate for $\Rb^{(1,2)}.$
          \begin{lem}\label{prop.eigenv.diff}
            Suppose that for the quadratic forms $\vb_1,\vb_2$, the conditions of Lemma \ref{Lem.pert.eigenv} are fulfilled, 
            \begin{equation*}
              \nb^{\sup}(\Ab^{-\frac12}(\Tb_1-\Tb_2)\Ab^{-\frac12},\theta)\le \mF.
            \end{equation*}
            and
            the operators $\Tb_{\io}=\Fb_{\io}\g\Ab^{-\frac12}$ are compact. Then 
            
            \begin{equation*}
              \nb^{\sup} (\Rb^{(1,2)},\theta)=\mF.
            \end{equation*}

          \end{lem}
\begin{proof}The result follows immediately from the relation $\nb^{\sup}(\Zb_j,\theta)=0,$ which is already established in the proof of Lemma \ref{Lem.pert.eigenv}.
\end{proof}

Note that this property imposes restrictions not only upon the difference $\Tb_1-\Tb_2$, but also upon the operators $\Tb_1,\Tb_2$ separately.

\section{Elliptic operators with singular perturbations}
\label{Sec.Sing.estim}
\subsection{Singular form-perturbations} 

 Let $\Om\subseteq\R^\Nb$ be a connected domain and $\Si$ be a compact subset in $\overline{\Om}$; let   $\m$ be a Radon measure supported in $\Si$.  Further, let $\Lc$ be a second order elliptic operator in $\Om$ with smooth coefficients. The corresponding quadratic form  is
 \begin{equation}\label{qform.concrete}
   \ab[u]=\int_{\Om}\left[\sum_{j,k}a_{jk}(X)\partial_j u\overline{\partial_k u}+t|u(X)|^2\right]dX, \, t>0,
 \end{equation}
 (where the zero order term is added in order to grant positivity)
 with domain $\dF[\ab]=H^1(\Om)$; this is the \emph{Neumann} realization of $\Lc$;  if needed, this will be reflected in the notation of the quadratic form, $\ab^{\Ns}$, and the same for the corresponding operator $\Ab=\Ab^{\Ns}.$                                                
 
 Along with $H^1(\Om),$ we will consider higher order Sobolev spaces $H^l(\Om),$ with integer  $l>1$ as well as, for a domain $\Om$ with compact complement, Hilbert spaces $\HF_{\Lc}^l(\Om)$ which locally, near $\partial \Om,$ are contained in $H^l$
 
  Under certain regularity conditions (specified later on) the restriction $\g$ of functions in the Sobolev space $H^l(\Om)$ to $\Si$ is a well defined bounded operator from the Sobolev space $H^l(\Om)$ to the Hilbert space $\GF=L_{2,\m}(\Si)$ with respect to the measure $\m,$ in particular, as explained in Sect.\ref{Sect.Sing.sets},  Hau{\ss}dorff measure $\Hc^d$ of the proper dimension $d\le\Nb$. Namely, this restriction, the operator $\g$, defined initially on $C(\overline{\Om})\cap H^l(\Om),$ can be extended by continuity to the whole of $H^l(\Om).$ 
  
   For a given real-valued function $V$ on $\Si,$ measurable with respect to the measure $\m$ on $\Si$,  the function $F$ is defined as $F=|V|^{\frac12}$, a real $\m$- measurable function on $\Si$. The bounded function $U(X)$ on $\Si$ is defined as the sign of $V(X)$: $U(X)=\sgn V(X):=V(X)|V(X)|^{-1}$ if $V(X)\ne 0$ and $U(X)=0$ otherwise.  In this setting,  the perturbing quadratic form $\vb[u]$ is defined as
 \begin{gather*}
   \vb[u]=\int_{\Si}V(X)|u(X)|^2\m(dX)=(F\g u, U F\g u)_{L_2(\Si,\m)}:=\\\nonumber(\Fb\g u, \Ub \Fb\g u)_{\GF}, \, u\in H^1(\Om), \GF=L_{2,\m}.
 \end{gather*}
with $\Fb,\Ub$ being the operators of multiplication by $F,U$; this is a concrete realization of the abstract scheme in Sect. \ref{Sect.Abstr.Pert}. In order to use this scheme, we need to check the abstract properties required in  Sect. 2 for the particular setting, which will be done in this Section.
 
 In this paper, we consider the   measure $\mu$ having compact support. We will show that
 if the 'density' $V$ belongs to a certain $L_{p,\m}$ class (with $p$ depending on $d, \Nb, l$) or, in some cases, the Orlicz class, then the    conditions \eqref{ab.bounded} and \eqref{qform6} are fulfilled and the above construction fits into the abstract setting of Section \ref{Sect.Abstr.Pert}.

We consider two  different cases in the study of the spectral properties of the difference of resolvents. On the one hand   $\Si$ may be a compact set in the boundary $\Si\subset\G=\partial \Om$, of positive, or possibly, zero surface measure and of Hau{\ss}dorff dimension $d\in(\Nb-2,\Nb-1]$. The domain $\Om$ may be bounded or having bounded complement. Here, the perturbation $\Vb$ consists in setting Robin type boundary conditions on $\Si,$ with measure $V\m$ serving as weight, while keeping intact the Neumann boundary conditions on the   remaining part of $\G$.  In particular, if $V=0$ everywhere on $\G,$ the quadratic form $\ab_{\vb}$ coincides with \eqref{qform.concrete}. Such, unperturbed, operator is the operator of the Neumann problem, $\Ab=\Ab^{\Ns}$. It turns out that the eigenvalue estimates for $(\Ab^{\Ns})^{-1}-(\Ab_\Vb)^{-1}$ have order depending on the Hau{\ss}dorff dimension of $\Si,$ in the special case when $d=\Nb-1$ for $\Si$ with  positive surface measure, the eigenvalue asymptotics holds.  These  results are valid  for the difference of resolvents of two Robin type problems  as well.
 
 The second case concerns  $\Om=\R^{\Nb}$ and  $\Si\subset \R^{\Nb}$ being a compact set of Ha{u\ss}dorff dimension $d\in(\Nb-2,\Nb]$  with    a Hau{\ss}dorff $\Hc^d$-- measurable function $V$ on $\Si.$ Here, we consider the difference of resolvents of Schr{\"o}dinger like operators with, possibly singular (for $d<\Nb$), potentials. In particular, the case fits in this picture when the set $\Si$ is a Lipschitz hypersurface, and the perturbation $\Vb$ consists, in fact, in setting 'transmission type' condition on $\Si$.
 In this way, we include here  $\delta$-interactions on a hypersurface, $d=\Nb-1$, systematically studied in the literature, see, e.g., \cite{BeExL}, but we fail to cover $\delta'$-type interactions, since this perturbation does not fit into our quadratic form approach. We obtain spectral estimates for the difference of powers of the resolvents as well and its eigenvalue asymptotics.
 
 Further on, we  specify the properties of the set $\Si$ and the weight function $V$, which grant that the abstract conditions in Section \ref{Sect.Abstr.Pert} are fulfilled.

\subsection{Singular sets and measures}\label{Sect.Sing.sets} We recall the definition of  Ahlfors regular sets and measures.
 Let $\Si$ be a compact set in $\R^\Nb$ and $\m$ be a Radon measure supported in $\Si$. For $d\in (0,\Nb]$, the measure $\m$ is called \emph{Ahlfors regular}  of dimension $d$ (sometimes the word '$d$-set'  is used for $\Si$) if for any $r\in(0,\diam(\Si)]$ and any point $X\in \Si,$
 \begin{equation}\label{Ahlf}
   \Ac_-r^{d}\le  \m(\Si\cap B(X,r))\le \Ac_+r^{d},
 \end{equation}
 with some positive constants $\Ac_{\pm}=\Ac_{\pm}(\Si)$, $B(X,r)$  denoting the ball with center at $X$ and radius $r$.
 
 This notion is widely used in geometric measure theory and many other topics in Analysis, see, e.g., \cite{David}, \cite{Tr}. It is known that such measure is equivalent to the  $d$-dimensional Hau{\ss}dorff measure $\Hc^d$ on $\Si$.  If the measure $\m$ satisfies \eqref{Ahlf} we will write $\m\in \AF^d,$ or $\Si\in\AF^d.$ 
 
 For an integer dimension $d$, the most obvious example, the leading one in this paper, is a compact Lipschitz surface  or a finite union of such surfaces. A more general construction, valid for fractional $d$ as well, is the following (we refer to the exposition in \cite{Tr}, Ch.4,5). 
 
 Self-similar measures are multi-dimensional generalizations of the well-known Cantor sets with Hau{\ss}dorff measure of a proper dimension. An affine mapping $S$ in $\R^{\Nb}$ is called \emph{similitude} with ratio $\ro$ if it is a composition of a translation, a rotation, and a homothety with coefficient $\ro<1$. Let us have a finite collection of similitudes $\Sc=\{S_j\}, 
 j=1,\dots,N$, with ratios $\ro_j.$ The compact set $\Kc$ is 
called invariant with respect to $\Sc$ if $\Kc=\cup_{j}S_j\Kc$. For any such collection of similitudes, there exists a unique invariant compact set $\Kc$. This set, for every fixed $m$, satisfies
\begin{equation*}
  \Kc=\bigcup_{\jb=(j_1,\dots,j_m)}S_{j_1}\dots S_{j_m}\Kc.
\end{equation*}
 Let $d$ be the Hau{\ss}dorff dimension of $\Kc.$ The set $\Kc$ is called \emph{selfsimilar},  if, moreover,  
\begin{equation}\label{Frakt.2}
  \Hc^d(S_j\Kc\cap S_{j'}(\Kc))=0, \,\, j\ne j'.
\end{equation}
 A strengthening of \eqref{Frakt.2}  is the \emph{open set condition}: there exists an open set $\Oc$ containing $\Kc$ such that 
 \begin{equation*}
   S_j(\Oc)\cap S_{j'}(\Oc)=\varnothing, j\ne j', \,  \, \mbox{and}\, \bigcup_j S_j(\Oc)\subset \Oc.
 \end{equation*}
 The property of our interest is that  if the open set condition is satisfied, then the Hau{\ss}dorff dimension $d$ of $\Kc$ equals the only positive solution of
  \begin{equation*}
   \sum_j \ro_j^d=1,
 \end{equation*}
and the Hau{\ss}dorff measure $\Hc^d$ restricted to $\Kc$ belongs to $\AF^d$. Obviously, a bilipschitz image of a $d$-set is a $d$-set again. More examples one can get by considering a disjoint union of $d$-sets, finite, and under some conditions, infinite.  
  
  It was found in \cite{RTsing} that sometimes, when studying spectral estimates, only one of the  inequalities in \eqref{Ahlf} is needed. In this connection,
  we call the measure (the set) lower-, resp., upper-regular of order $d$ if the left, resp., the right inequality in \eqref{Ahlf} is satisfied. The corresponding classes of measures will be denoted $\AF^d_-,$ resp., $\AF^d_+$, so $\AF^d=\AF^d_+\cap\AF^d_-$. (If the support of the measure $\m$ is not compact, the  inequalities in \eqref{Ahlf} are supposed to hold for all $r>0$, we, however, restrict our considerations to compact sets.)
\subsection{Regularity of the boundary}For different types of spectral problems under consideration, we will require different   regularity level  of the domain $\Om\subset\R^\Nb.$ For the Schr{\"o}dinger type  problems, where only the local regularity of solutions of the elliptic equations is used, $\Om$ is an arbitrary domain without additional restrictions, and we can consider the whole $\R^\Nb$. For boundary problems with Robin boundary conditions, when considering the difference of resolvents, we suppose that the boundary is present,  is compact and is of class $C^{1,1},$ this means that the boundary can be represented, locally  in properly rotated local co-ordinates $X=(x_1,x')$,  as the graph of a function $x_1=\psi(x')$ where first order partial derivatives of the function $\psi$ belong to the Lipschitz class. In the latter case, the following  local regularity property plays the key role.
\begin{proposition}\label{Prop.H 2.2} Let $\Om$ be a domain of class $C^{1,1}$  in $\R^\Nb$ and let  $\Lc$ be a Hermitian  elliptic second order operator in $\Om$ with smooth coefficients. Let $\Om'\subset\Om$ be a bounded domain with $C^{1,1}$ boundary such that the boundary of $\Om$ lies in $\overline{\Om'}$. Then for any $f\in L_2(\Om),$ the (unique) solution $u$ of the Neumann boundary problem  
  \begin{equation}\label{Neumann}
    \Lc u+tu=f, \, \, \partial_{\n(\Lc)} u_{|_{\partial \Om}}=0, t>0,
  \end{equation}
  belongs to the Sobolev space $H^{2}(\Om')$ and 
  \begin{equation*}
    \|u\|_{H^2(\Om')}\le C \|f\|_{L^2(\Om')}.
  \end{equation*}
\end{proposition} 
For a smooth boundary, this is the general elliptic regularity fact, presented in a number of standard sources. For the boundary of finite smoothness, as above, this result can be found, with different degree of detalization, in certain advanced treaties on elliptic boundary problems, sometimes in the version for the Dirichlet boundary conditions, with referral to a similar reasoning for the Neumann conditions. We can cite here, e.g.,   \cite{Grisvard}, Theorem 2.5.1.1.

 For higher order operators, we suppose, for the sake of brevity, that the boundary is infinitely smooth. For a bounded domain $\Om,$ the general elliptic regularity results, see, e.g., \cite{LiMag}, imply that the operator $(\Ab^{\Ns}+t)^{-l/2}$ acts continuously from $L_2(\Om)$ to $H^{l}(\Om)$. For an \emph{exterior} domain $\Om$ with compact boundary, some more discussion is needed.
 
 We suppose that the operator $\Lc$ is uniformly elliptic in $\Om$ and has coefficients with bounded derivatives of all orders. Let $\Ab^{\Ns}$ be the Neumann operator for $\Lc+t$ in $\Om.$ The domain of $(\Ab^{\Ns})^{\frac12}$ is the Sobolev space  $H^1(\Om).$ The domain of a higher power, $(\Ab^{\Ns})^{\frac{l}2}$, of our operator  (we denote this domain by $\HF_{\Lc}^{l}$, with graph norm),   may be hard to describe explicitly as soon as it concerns the behavior at infinity. Locally, as follows from the elliptic regularity, functions in $\HF_{\Lc}^{l}$ belong to $H^l.$ A detailed analysis of the boundary regularity of functions in  $\HF_{\Lc}^{l}$ is performed, e.g., in \cite{Grubb.Ext}, \cite{Malamud} (see, especially, Sect.4, Proposition 4.9 in the latter paper.) It follows, in particular, that the functions in $\HF_{\Lc}^{l}$ belong to $H^l$ near the boundary. These results can be summarized as the following.
\begin{proposition}\label{boundary.regil.prop}Let $\Om\subset\R^\Nb$ be a domain with smooth compact boundary $\G=\partial\Om$. Let $\Om'$ be a bounded subdomain in $\Om$ such that $\G\subset\overline{\Om'}$. Then the operator $\Jb$ of restriction of functions in $\HF_{\Lc}^{l}$ to $\Om'$ acts boundedly to $H^l(\Om').$ 
\end{proposition}
It follows, in particular, that the operator $(\Ab^{\Ns})^{-l/2}$ acts as a bounded operator from $L_2(\Om')$ to $H^{l}(\Om').$

Another fact, important in our study, is the possibility of continuation of solutions of the equation \eqref{Neumann} from  $\Om$ across the boundafy, with preservation of smoothness. 
\begin{lem}\label{lem.contin}
  If $\Om$ is a  $C^{0,1}$ domain with compact boundary then there exists a bounded extension operator $\Jb:H^l(\Om)\to H^l(\R^\Nb)$ for all $l$.
\end{lem}
This is a particular case of the classical  Stein extension theorem, see Theorem 5, Chapter 6, in \cite{Stein}.

Thus, for a smooth boundary, under the conditions of Proposition \ref{Prop.H 2.2}, there exist a continuation operator $\Jb$ in $H^l$ for all  solution of the Neumann problem.

\subsection{Singular quadratic forms in the space $H^l$}\label{Subs.form}
The conditions we impose on the functions defined on singular sets will depend on the exponent $\theta>0. $ For the sake of the uniformity of formulations, we introduce the following notation.
For a given exponent $\theta$ and a finite  measure $\m$, we denote by $L_{(\theta),\m}$ the space $L_{\theta,\m}$ for $\theta>1$, $L_{1,\m}$ for $\theta<1$, and, for $\theta=1$, the Orlicz space $L_{\Psi,\m}$ corresponding to the function $\Psi(s)=(s+1)\log(s+1)-s:$ $\int \Psi(|V(X))|\m(dX)<\infty.$ By $\|F\|_{(\theta),\m}$ we denote the norm of a function $F$ in this space.
There are several equivalent ways to introduce a norm in the Orlicz space; we will use the Luxemburg norm (see, e.g., Chapter 8 in \cite{Stein}).

The statement to follow presents some basic, mostly known, facts about the properties of the quadratic form 
\begin{equation}\label{form.details}
  \ab_{V\m}[u]=\int V(X)|u(X)|^2 \m(dX)
\end{equation}
considered in the Sobolev spaces $H^l(\R^\Nb)$. We will use it for integer  values of $l$, while $l=1$ and $l=2$ are the most important ones.

First of all, if $\Nb<2l$, the Sobolev space $H^{l}(\R^\Nb)$ is embedded in $C(\R^\Nb),$ therefore the quadratic form $\ab_{V\mu}[u]$ is well defined and bounded in $H^l(\R^\Nb)$ for any finite measure $\m$ and for any $V\in ,L{1,\m}.$ The case $\Nb\ge 2l$ is considerably more delicate.

\begin{lem}\label{Prop.bounded}
  Let $\Nb\ge 2l$ and let the measure $\m$ belong to $\AF_+^d$, for some $d$,\\ $\Nb-2l<d\le \Nb$ if $\Nb>2l$, $\m\in\AF^d$, $d>0$ if $\Nb=2l$. Suppose that $V\in L_{(\theta),\m}$, $\theta=\frac{d}{d-\Nb+2l}$. Then 
  \begin{enumerate}\item the quadratic form \eqref{form.details},
  defined initially for $u\in H^{l}(\R^\Nb)\cap C(\R^\Nb)$, is bounded in the $H^l(\R^\Nb)$-metric and admits a unique bounded extension to the whole Sobolev space  $H^l(\R^{\Nb});$
\begin{equation}\label{Bddness.inside}
    \left|\int |u(X)|^2 V(X)\m(dX)\right|\le C_0\|V\|_{L_{(\theta),\m}}\|u\|^2_{H^l(\R^\Nb)}, \, u\in H^l(\R^\Nb);
  \end{equation}
  \item the quadratic form
   \begin{equation}\label{Qform3}
    \ab_{V\m}=\int_{\R^\Nb}|\nabla^l u(X)|^{2} dX +\int|u(X)|^2V(X)\mu(dX), u\in H^l(\R^\Nb),
\end{equation}
 is lower semibounded in $H^l(\R^\Nb)$:      
      \begin{gather}\label{semibdd}
        \int_{\R^\Nb}|\nabla^l u(X)|^{2} dX +\int|u(X)|^2V(X)\mu(dX)+t\int_{\R^\Nb}|u(X)|^2 dX\ge\\\nonumber c\int_{\R^\Nb}|\nabla^l u(X)|^{2} dX, \, u\in H^l(\R^\Nb)
      \end{gather}
      for some $t\in\R^1$.
      \end{enumerate} 
      The same statements are valid, with $\R^\Nb$ replaced in the formulation by a domain $\Om$ with compact $C^{0,1}$ boundary.    
\end{lem}
\begin{proof} It suffices to prove Lemma for the whole space; the case of a domain follows immediately by using the Stein extension theorem. Statement (1) can be found, e.g., in \cite{RSh} for the case $\Nb=2l$ and in \cite{RTsing} for $\Nb>2l$; both are direct consequences of trace theorems in the book by V. Maz'ya, \cite{MazBook}.

After this, the starting point in the proof of Statement (2) is the classical multiplicative inequality in the Sobolev space.
 We reproduce its version as a particular case of Theorem 1.4.7 in \cite{MazBook}:\\
\emph{Let the measure $\m$ belong to $\AF^d_+$ with $d>\Nb-2l.$ Then for any function $u$ in $\Ds(\R^\Nb)$,
\begin{equation}\label{Maz.emb}
\|u\|_{L_{q,\m}}\le C_0\|\nabla^l u\|_{L_2(\R^\Nb)}^\tau\|u\|_{L_2(\R^\Nb)}^{1-\tau},\,q\ge 2,\, \frac{d}{q}>\frac{\Nb}{2l}-1, \tau=\frac{\Nb}{2l}-\frac{d}{q}.
\end{equation}
with constant $C$ depending on the measure $\m$ and parameters $\Nb, q, d.$ }

It follows from \eqref{Maz.emb} that for any $\e>0$, there exists $C_\e$ such that 
\begin{equation}\label{Maz.emb.sum}
\|u\|^2_{L_{q,\m}}\le \e\|\nabla^l u\|^2_{L_2(\R^\Nb)}+C_\e\|u\|^2_{L_2(\R^{\Nb})}. 
\end{equation}
Now, for proving  \eqref{semibdd}, for a given $V\in L_{\theta,\m},$ $\theta=\frac{d}{d-\Nb+2l}$ and given $\de>0$, we find such  $h=h(\de)>0$ so that
\begin{equation*}
  V=V_1+V_2,\, |V_1(X)|\le h, \|V_2\|_{\theta,\m}<\de.
\end{equation*}
Correspondingly, the quadratic form  $\ab_{V\m}$ splits into the sum,
$\ab_{V\m}[u]=\ab_{V_1\m}[u]+\ab_{V_2\m}[u].$
For $\ab_{V_1\m}[u],$ we apply \eqref{Maz.emb.sum} with $q=2,$ $\tau=\frac{\Nb-d}{2l}$, which gives
\begin{gather*}
      |\ab_{V_1\m}[u]|=
      \left|\int V_1(X)|u(X)|^2 \m(dX)\right|\le \\\nonumber
      h\|u\|_{L_{2,\m}}\le h\e \int_{\R^\Nb}|\nabla^l u|^2dX+ hC_\e\int_{\R^\Nb}  |u(X)|^2 dX.                                                     \end{gather*}
To estimate  the second term, we apply \eqref{Maz.emb.sum} for $q=\frac{2d}{\Nb-2l},$ i.e., $\tau=1$. We obtain, by the H{\"o}lder inequality
\begin{equation*}
  |\ab_{V_1\m}[u]|\le \|V_2\|_{L_{\theta,\m}}\|u\|^2_{L_{q,\m}}\\\nonumber
 \le  C\de\int_{\R^\Nb}|\nabla^l u(X)|^2dX.
\end{equation*}
We take $\e,\de$ so small that $a\e<\frac14,$ $C\de<\frac14,$
therefore,
\begin{equation*}
  \left|\int V(X)|u(X)|^2\mu(dX)\right|\le \frac12\int_{\R^\Nb}|\nabla^l u|^{2}dX+C'\|u\|^2_{L_2(\R^\Nb)},
\end{equation*}
and, finally,
\begin{equation*}
  \int_{\R^\Nb}|\nabla^l u|^{2}dX+\int V(X)|u(X)|^2\mu(dX)\ge \frac12\int_{\R^\Nb}|\nabla^l u|^{2}dX-C'\|u\|^2_{L_2(\R^{\Nb})},
\end{equation*}
which justifies the semiboundedness inequality.

In dimension $\Nb=2l,$ a similar reasoning proves the required semiboundedness of the quadratic form \eqref{Qform3},  with the weight $V$ belonging to the Orlicz class $L_{\Psi,\m}.$
The trace inequality replacing  \eqref{Maz.emb}, \eqref{Maz.emb.sum} is Theorem 11.8 in \cite{MazBook} and its $H^l$  version, see detains in \cite{RSh}, \cite{GRInt}.
\end{proof}

In the next section, we  will use these results for $l=1$ and $l=2$. The results for higher order will be used in the treatment of the difference of powers of resolvents. 
\subsection{A singular measure on the boundary and the operator for the Robin problem}\label{Robin.Setting}
Here we discuss the same kind of inequalities for the measure supported at the boundary of $\Om.$
We recall that the  classical Robin boundary problem in a domain $\Om\subset\R^{\Nb}$ for the elliptic operator $\Lc$ is
\begin{equation}\label{Robin.Classical}
  \Lc u +t u=f \, \mbox{in}\, \, \Om, \, \partial_{\n(\Lc)} u(X)+V(X)u(X)=0\,\, \mbox{on}\,\, \partial\Om,
\end{equation}
with a 'nice' function $V.$ This problem can be expressed in the variational form, see, e.g.,  Sect.5 in \cite{ArendtWarma}, see also \cite{Malamud}.   The operator which maps $u$ to $f$ is defined by the quadratic form
\begin{gather}\label{RobinOperator}
  \ab_{R,V}=\int_{\Om}\langle\aF(X)\nabla u(X),\nabla u(X)\rangle dX +t\int_{\Om}|u(X)|dX+\\\nonumber\int_{\partial \Om} V(X)|(\g u)(X)|^2 \sigma(dX), u\in H^1(\Om),
\end{gather}
where $\aF(X)$ is the matrix of leading coefficients of $\Lc$ and $\sigma$ is the natural surface measure on $\partial\Om.$ 
The quadratic form  is considered on the Sobolev space $H^1(\Om)$ equipped with the metric defined by the first two terms in \eqref{RobinOperator}. Note that for $u\in H^1(\Om)$, the trace $\g u$ on the boundary is well defined as an element in $L_2(\partial\Om)$ and even as an element in the Sobolev space $H^{\frac12}(\partial\Om)$, as soon as the boundary is Lipschitz. It is in this variational form that we will  introduce the generalization of the classical Robin problem  \eqref{Robin.Classical}.
In  more recent papers, a more general setting was considered, with the function $V$ in \eqref{Robin.Classical} replaced by a  pseudodifferential operator $\Theta$ on $\partial\Om$, acting in some convenient way between Sobolev spaces on  $\partial\Om$, see, e.g., \cite{BLLLP}, \cite{BLL}, \cite{BLL2}, \cite{Malamud}. In our case, this operator will be replaced by a measure on $\partial\Om,$ possibly, singular with respect to $\s$, with summability conditions imposed.

Namely, having a measure $\m$ supported in $\G=\partial\Om$ and a $\m$-measurable function  $V$, we can consider the quadratic form $\ab_{V\mu}[u]=\int V(X)|\g u(X)|^2\mu(dX)$ instead of the one on the second line in   \eqref{RobinOperator}.

For such quadratic form, the results similar  to the ones in the previous subsection are valid, under the condition that the boundary $\G=\partial\Om$ possesses the $H^l$-extension property; in our case, recall, we suppose that it is a Lipschitz one.
\begin{lem}\label{Lem.boundd.boundary.2}Let the measure $\mu$ be  supported in $\G$. Suppose that the conditions of Lemma \ref{Prop.bounded} are satisfied, just with the inequality $\Nb-2l<d\le\Nb$ replaced by $\Nb-2l<d\le\Nb-1$. Then the statements in  Lemma \ref{Prop.bounded} are correct. \end{lem}

\begin{proof}We consider  the bounded \emph{extension} operator $\Jb:H^l(\Om)\to H^l(\R^\Nb),$ $\|\Jb\|=C_1$.
 and obtain, using the fact that $\ab_{V\m}[\Jb u]=\ab_{V\m}[u]$
  \begin{equation}\label{est.cont}
   \vb_{V\mu}[u]\le C_0\| \Jb u\|^2_{H^l(\Om')}\le C_0C_1\|u\|^2_{H^l(\Om)},
 \end{equation}
where $C_0$ is the constant in  \eqref{Bddness.inside}. 
 The same reasoning takes care of the lower semiboundedness of the quadratic form $\ab_{V\m}.$
\end{proof}.


 \section{Spectral estimates for singular  measures}\label{estimates}
 In this section, we adapt  the results of the papers \cite{RSh} and \cite{RTsing} to our concrete setting. We  consider two main cases, $l=1$, corresponding to quadratic forms in the Sobolev space $H^1(\Om)$, and $l=2$, corresponding to the Sobolev space $H^2(\Om)$. Then we discuss also the case of a general integer $l$, used later to treat the difference of powers of resolvents.

 \subsection{Upper spectral estimates}\label{Sub.upper}
 In the papers \cite{RSh}, \cite{RTsing}  spectral estimates of weighted pseudodifferential operators with singular measure acting as weight were obtained. Let $\mu$ be a measure with support (the smallest closed set of full $\mu$-measure) $\Si\subset\overline{\Om}\subseteq \R^\Nb$ and let  $V(x)$ be a $\m$-measurable function on $\Si$ (supposed, if necessary, being extended by zero to the whole of $\Om$ or to the whole of $\R^\Nb$); we denote by $\pmb{\pi}$ the  measure $V\m$.
 
First, we  will obtain   spectral estimates for an operator ${\Sbb}={\Sbb}(\pmb{\pi})={\Sbb}(\pmb{\pi}, l,\Om)$ defined by the quadratic form $\vb[u]=\int |u(X)|^2 \pmb{\pi}(dX)$ in the Sobolev space $H^l(\Om).$ The decisive property of these estimates is that they depend not on the $L_\infty$ norm of the weight function $V$, but on  its  norm in some integral metric. Since the groundbreaking papers by M. Birman and M. Solomyak, such estimates keep being, probably, the only way to prove spectral \emph{asymptotics} for operators with singularities.

 The boundedness of $\Sbb$, equivalently, the boundedness of the quadratic form $\vb[u]$ in $H^l(\Om)$, was established in Sect. \ref{Sec.Sing.estim}. Under the conditions below, the operator $\Sbb$ turns out to be compact. We will use the notations set in Sect. \ref{Sect.Abstr.Pert} for the characteristics of the distribution of the spectrum of $\Sbb$.
 
  The following results were established in \cite{RSh}, \cite{RTsing}.
The conditions imposed on the measure  look differently for the following  three cases:
\begin{gather}\label{Conditions123}
\left\{
  \begin{array}{ll}
    \mbox{subcritical}, & { 2l<\Nb,\, \mu\in \AF^d_+, \, d>\Nb-2l;} \\
    \mbox{critical}, & {2l=\Nb,\,\m\in \AF^d,\,0<d\le\Nb;} \\
    \mbox{supercritical}, & { 2l>\Nb,\, \m\in\AF^d_-,\, 0<d\le \Nb.}
  \end{array}
\right.
\end{gather}

\begin{thm}\label{Thm.estimates}\emph{[}Theorem 3.3  in \cite{RSh}, Theorem 3.3 and Theorem 3.9 in \cite{RTsing}\emph{]} and let $\theta=\frac{d}{d+2l-\Nb}$.  Let the weight function $V$ belong to $L_{(\theta),\m}$. Then for the corresponding case in \eqref{Conditions123} above, for the operator $\Sbb$, the following estimates hold:
\begin{gather}\label{Thm.Estim.ineq}\nb^{\sup}(\Sbb,\theta)\le C \|V\|_{(\theta),\m}^{\theta},\\\nonumber
\mbox{and}\\\nonumber
\nb_{\pm}^{\sup}(\Sbb,\theta)\le C\|V_{\pm}\|_{(\theta),\m}^{\theta},
\end{gather}
  \end{thm}
\noindent with constant $C$ not depending on $V$ in the corresponding space.
In the sources cited above,  the dependence of the constants in the spectral estimates on the measure $\m$ and, especially, on the constants in \eqref{Ahlf}, are described as well, however, since the measure in our considerations is fixed,  we do not need such results here. In the critical case, in \cite{RSh}, a different, 'averaged' norm in the Orlicz space is used. Using the Luxemburg norm instead changes  only the constant $C$ in \eqref{Thm.Estim.ineq}.

 Eigenvalue estimates in Theorem \ref{Thm.estimates} extend easily to operators $\Sbb(\pmb{\pi},l,\Lc,\Om)$ generated by the quadratic form $\vb[u]$ in the space $\HF_{\Lc}^{l},$ the domain of the operator $\Ab^{l/2},$ for the case when the measure $\m$ is supported strictly \emph{inside}  $\Om.$
 \begin{cor}\label{cor.Thm.Est} Let, in conditions of Theorem \ref{Thm.estimates}, $\Lc$ be a second  order elliptic operator in $\Om$ with smooth coefficients. Then for the operator $\Sbb(\pmb{\pi},l,\Lc,\Om)$ spectral estimates \eqref{Thm.Estim.ineq} hold.
 \end{cor}
\begin{proof}We use Lemma \ref{BSLem.extension} and its corollary. First, we take a bounded domain $\Om_1\subset\Om$ with smooth boundary, $\supp\pmb{\pi}\subset\Om_1$. By the local elliptic regularity, the operator $\Jb_1$, the restriction of functions in  $\HF_{\Lc}^{l}$ to $H^l(\Om_1)$ is bounded. Next, the Stein  operator $\Jb_2$ of extension from $H^l(\Om)$ to $H^l(\R^\Nb)$ is bounded as well. By Corollary \ref{cor.restr},  the spectral estimates survive, probably, with some controlled degradation in the constants.
\end{proof}

We recall here that the eigenvalue estimates for the operator $\Sbb(\pmb{\pi},l,\Lc,\Om)$ are equivalent to eigenvalue estimates for the operator 
\begin{equation*}
  \Wb_l=(\Fb\g\Ab^{-\frac{l}{2}})^*\Ub(\Fb\g\Ab^{-\frac{l}{2}}),
\end{equation*}
see Sect.2. These estimates will be discussed in the next section.

\subsection{Estimates for measures supported at the boundary of $\Om$}\label{Sect.est.bdry.}
In the study of the Robin boundary problem, we will need eigenvalue estimates for the case when the measure $\m$ is supported in the boundary of the domain $\Om\subset\R^{\Nb}$, $\m\in\AF^{d}$ with $\Nb-2<d\le\Nb-1$. 

For a bounded domain $\Om\subset\R^\Nb$ with boundary of class $C^{1,1}$,  for the elliptic operator $\Lc$ with coefficients  smooth in $\overline{\Om}$, we know that the domain of $(\Ab^{\Ns})^{\frac12}$ coincides with $H^1(\Om)$ and the domain of $\Ab^{\Ns}$ coincides with $H^2(\Om).$ Therefore, the eigenvalue estimates for the operator defined by the quadratic form $\vb$ in these Hilbert spaces are the same as estimates in the corresponding Sobolev spaces, i.e., are the same as in Theorem \ref{Thm.estimates}. For higher values of $l,$ we require that the boundary $\G$ is smooth, and the reasoning goes in the same way, using the elliptic regularity.

For an unbounded domain $\Om$ with compact boundary, we use the boundary elliptic regularity, described in Sect. 3.3. The functions in the domain of $(\Ab^{\Ns})^{\frac{l}2}$ belong to $H^l$ in a neighborhood of the boundary, and we can repeat the reasoning in Corollary \ref{cor.Thm.Est}, where Corollary \ref{cor.restr}  is used again, reducing the estimation of eigenvalues of the operator $\Sbb(\pmb{\pi},l,\Lc,\Om)$ to estimates in the Sobolev space, found in Theorem \ref{Thm.estimates}.

\subsection{Lower estimates}\label{Sub.lower}We present here some lower spectral estimates for operators of the form  $\Sbb,$ with sign-definite, say, positive weight function $V$. Here, unlike the upper estimates above, the measure $\mu$ is supposed to be two-sided Ahlfors-regular, $\mu\in\AF^d,$ for all relations between $l$ and $\Nb$, therefore, we may suppose that $\m$ coincides with the Hau{\ss}dorff measure $\Hc^d$. We reproduce here the results established in  the book \cite{Tr}. The particular case we are interested  in is $l=2.$ For this case, Theorem 28.6 in \cite{Tr}, with $m=2,$ $\vk=4,$  gives the following result. We formulate it   in our present notations.
\begin{thm}[Triebel's lower estimate.]\label{Thm.TR.lower}
Let $\Om$ be a bounded domain in $\R^\Nb$, let $\Si$ be a compact set in $\Om,$ Ahlfors $d$-regular $\Nb-4<d\le\Nb$, $m=1,\vk=2$, $\m=\Hc^d$ on $\Si$.  Consider the operator $\Sbb=\Sbb(\m,2,\Om)$, defined by the quadratic form $\int_\Si |u(X)|^2 \Hc^{d}(dX)$ in the Sobolev space $H^2(\Om)$. Then
  \begin{equation*}
    n^{\inf}(\Sbb,\theta)>0, \theta = \frac{d}{4-\Nb+d}.
  \end{equation*}
\end{thm}
\begin{rem}\label{Rem.lower}
The equivalence of our formulation with H.Triebel's one is explained in \cite{Tr}, see (28.27) there.  Theorem \ref{Thm.TR.lower} does not provide estimates for operators with non-sign-definite $V$ and we refer to it only in order to give a que that our \emph{upper} eigenvalue estimates are order sharp.\\
 As it concerns \emph{upper} spectral estimates established in \cite{Tr}, the interested Reader is addressed to the Appendix in \cite{RTsing}, where the non-sharpness of these estimates is discussed, therefore, we used more sharp estimates in Theorem \ref{Thm.estimates} instead.\end{rem}

 \section{The resolvent of a singular Schr{\"o}dinger operator}\label{Sect.Shrod.}
 Here we apply the general results on the spectrum of singular quadratic form perturbations. 
 
\subsection{The  Schr{\"o}dinger operator with singular potential}\label{subsect.trans}
Let $\Om$ be an  unbounded domain in $\R^{\Nb}$ with $C^{1,1}$ boundary and compact complement. Suppose that $\Si$ is a compact subset in $\Om$ and let  $\mu$ be a Radon measure supported in $\Si$ and $V(X)$ be a   real $\mu$-measurable function. Let further on, $\Lc$ be a  uniformly  elliptic formally self-adjoint second order operator in $\Om,$ 
\begin{equation}\label{operator L}
  \Lc=-\sum_{j,j'}\partial_j a_{j,j'}(X)\partial_{j'} ,
\end{equation}
with bounded smooth coefficients $a_{j,j'}$  With the operator \eqref{operator L}, we associate the quadratic form
\begin{equation}\label{form a concrete}
  \ab[u]=\int_{\Om}\sum_{j,j'}a_{j,j'}(X)\partial_{j}u(X)\overline{\partial_{j'}u(X)}dX +t\int_{\Om}|u(X)|^2 dX.
\end{equation}
with $t$ large enough. This form, considered for $u\in H^1(\Om),$ defines the \emph{Neumann} operator $\Ab=\Ab^{\Ns}$ in $L_2(\Om)$. For $t>0$, this operator is positive in $H^1(\Om)$, so that $\ab[u] $ is equivalent to the squared norm in $H^1(\Om)$; we suppose that this is already done and keep our notations.

Let now the measure $\mu$ and the weight function $V$ satisfy the conditions of Lemma \ref{Prop.bounded} for $l=1$, $d\in(\Nb-2,\Nb].$ Denote by $\g $ the restriction operator,  bounded as acting from $H^1(\Om)$ to $L_2(\G)$.
By this Lemma,  the perturbing quadratic form $\vb[u],$
\begin{equation*}
  \vb[u]=\int_\Si V(X)|u(X)|^2\mu(dX)= (\Ub\Fb\g u, \Fb\g u )_{L_2(\Si, \m)}, \, u\in H^1(\Om),
\end{equation*}
  is bounded in $H^1(\Om),$ and defines in $H^1(\Om)$ a bounded operator ${\Sbb}_{\vb}$, for which the spectral estimates in Theorem \ref{Thm.estimates} are valid.
The operator $\Ab^{\frac12}$ is an isomorphism between $H^1(\Om)$ and $L_2(\Om)$.
Thus, setting $v=\Ab^{\frac12}u\in L_2(\Om),$ we obtain the quadratic form
\begin{equation*}
 \vb[u]= \vb[\Ab^{-\frac12}v]=(\Ub\Fb\g \Ab^{-\frac12}v,\Fb\g\Ab^{-\frac12}v)_{L_2(\Si, \m)}, v\in L_2(\Om).
\end{equation*}
This quadratic form defines the operator $\Tb=\Tb_{\vb}$ in $L_2(\Om)$, unitarily equivalent to ${\Sbb}$,
\begin{equation*}
 \Tb=\Tb_{\vb}=\left[\Fb\g\Ab^{-\frac12}]^*\Ub[\Fb\g\Ab^{-\frac12}\right].
\end{equation*}
 In this way, we have arrived to the setting of our Sect.2, with a compact operator $\Tb$ (Theorem \ref{Thm.estimates} with $l=1$ gives order sharp eigenvalue estimates, but at the moment, only compactness is needed.)
 
Following Sect.2, we will also need spectral estimates for the operator
\begin{equation}\label{oper.S concrete}
\Wb\equiv \Wb_{\vb,2}= \Ab^{-1/2}\Tb\Ab^{-\frac12}=\left[\Fb\g\Ab^{-1}\right]^{*}\Ub\left[\Fb\g\Ab^{-1}\right]
\end{equation}
in $L_2(\Om).$ This operator is defined by the quadratic form
\begin{equation}\label{form.s.concrete}
  \wb[v]=\vb[\Ab^{-\frac12}v]=\int U(X)|F(X)|^2 |(\Ab^{-1}v)(X)|^2 \m(dX), v\in L_2(\Om).
\end{equation}
By setting $\Ab^{-1}v=w,$ we transform \eqref{form.s.concrete} to
\begin{equation}\label{form2}
  \wb[v]=\sbb[w]= \int U(X)|F(X)|^2 |\g w(X)|^2 \m(dX), \, \, w\in \Dc(\Ab).
\end{equation}
The quadratic form \eqref{form2} is considered in the Hilbert space $\Dc(\Ab)=H^2_{\Lc}$, the domain of the self-adjoint operator $\Ab$.  For an unbounded $\Om,$ this domain, generally, is rather hard, if possible, to describe. This circumstance was taken care of  in Sect.4, by means of describing the local regularity of functions in $H^2_{\Lc}$

The corresponding result is just a re-formulation of Theorem \ref{Thm.estimates}. Before stating it,  we recall the restrictions imposed by this theorem on the measure $\m$ and the weight function $V$ for the most important values $l=1$ and $l=2$. These conditions should be satisfied both for $l=1$, in order to be able to define the operator $\Tb$ and grant its compactness, and for $l=2,$ in order  to estimate the eigenvalues of $\Wb:$ 
\begin{gather}\label{Condition 2}
\mu\in\AF^d,\,  0<d\le 2\, \mbox{for}\, \Nb=2;\\\nonumber
 \mu\in \AF^d,\,  \Nb-2<d\le \Nb,\, \mbox{for}\, \Nb=3,\, \mbox{or}\, \Nb=4;\\\nonumber   
  \mu\in\AF^d_+,\, \Nb-2<d\le \Nb,\, \mbox{for } \, \Nb>4,\\\nonumber
  \mbox{and}\\\nonumber  V\in L_{(\vartheta),\m}, \vartheta=\frac{d}{d+2-\Nb}.
  \end{gather}

\begin{thm}\label{Th.Est.Sb} Let the measure $\mu$ and the weight function $V$ satisfy the above conditions. Then for the operator $\Wb=\Ab^{-\frac12}\Tb\Ab^{-\frac12}$ the following estimate holds for $\theta =\frac{d}{d-\Nb+4}$
\begin{equation*}\label{ESt.Sbb}
  \nb^{\sup}(\Wb, \theta)\le C \|V\|_{L_{(\theta),\m}}^{\theta}, ,
\end{equation*}
\begin{equation*}
  \nb_{\pm}^{\sup}(\Wb, \theta)\le C \|V_{\pm}\|_{L_{(\theta),\m}}^{\theta},
\end{equation*}
\end{thm}
Similarly, the eigenvalue estimates for other types of problems, obtained in Sect.4 can be re-written in terms of 'Birman-Schwinger operators.' 
\begin{thm}\label{Est.Wb.Gen}(1) Let $\Om\subset\R^{\Nb}$ be an open set and $\m$ be a Radon measure compactly supported in $\Om$, $l\ge 1$ and let $\Ab=\Ab^{\Nc}$ be the Neumann realization of the second order elliptic operator $\Lc$ in $\Om. $ Suppose  that the measure $\m$ satisfies the condition \eqref{Conditions123} corresponding to the relation between $l$ and $\Nb$ and $V\in L_{(\theta),\m},$ $\theta=\frac{d}{d+4l-\Nb}$. 
Then for the operator
 
 \begin{equation}\label{W.definition}
 \Wb_l=\Wb_{l,V,\m}=\Ab^{-(l-1/2)}\Tb\Ab^{-(l-1/2)}=(F\g\Ab^{-l})^*U,(F\g\Ab^{-l}),\, F=|V|^{\frac12}, \, U=\sgn(V),
 \end{equation}
  the following estimate is valid:
\begin{equation}\label{MainEstl1}
  \nb^{\sup}(\Wb_l,\theta)\le C\|V\|_{(\theta),\m}^{\theta}
\end{equation}
(2) Let $\Om$ be a domain with compact $C^{1,1}$ boundary, for $l=1,2$, or a domain with compact smooth boundary, for $l>2$.  Let $\m$ be a measure on $\G=\partial\Om,$ satisfying the condition \eqref{Conditions123} corresponding to the relation between $l$ and $\Nb$ and $V\in L_{(\theta),\m},$ $\theta=\frac{d}{d+2l-\Nb}$. 
Then for the operator $\Wb_l$, the estimate \eqref{MainEstl1} is valid.
\end{thm}

Theorem \ref{Est.Wb.Gen} implies spectral estimates for some non-self-adjoint operators arising in the study of the difference of  powers of resolvents.

First, we consider the operator $\Qb_l=\Qb_{l,G,\m}$, for an integer $l>0$ or $l=\frac12$.
\begin{equation*}
  \Qb_l u= G\g \Ab^{-l} u; \Qb_l:L_2(\R^\Nb)\to L_{2,\m},
\end{equation*}
where $\Ab$ is a second order elliptic operator in $\R^\Nb$ and $\m$ is a measure with compact support,  $\m\in\AF^d.$
\begin{lem}\label{lemQ}
  Let $G\in L_{(\y),\m}, $ where $\y=\frac{2d}{d-\Nb+4l}$. Then
  \begin{equation}\label{est.Q}
    \nb^{\sup}(\Qb_l, \y)\le  C\|G\|_{(\y),\m}^{\y}.
  \end{equation}
\end{lem}
\begin{proof} Let $l$ be an integer, first. Since for the operator $\Wb_l$ in \eqref{W.definition}, $\Wb_l=\Qb_{l}^*\Qb_{l}$ with $V=|G|^2,$
the singular numbers of $\Qb_l$ ar square roots of the singular numbers of $\Wb_l$, and the required estimate follows immediately from \eqref{MainEstl1}. If $l=\frac12$, the same reasoning goes through, using $\Qb_{l}^*\Qb_{l}=\Tb$ and applying the estimate in Theorem \ref{Thm.estimates}.
\end{proof} 

Next we need a spectral estimate for the operator
\begin{equation}\label{Q12}
 \Zb=\Zb_{l_1,l_2,G_1,G_2,\m}=(\Qb_{l_1,G_1,\m})^*(\Qb_{l_2,G_2,\m}), \, l_1, l_2\ge 1,
\end{equation}
this means for the product of two operators in Lemma \ref{lemQ}, with different $l, $ again with an integer $l_{\io}$ or $l_{\io}=\frac12.$
We apply Lemma \ref{lemQ} to each factor in \eqref{Q12} and then the Ky Fan inequality \eqref{kFmult8}, which gives the following estimate.
\begin{lem}\label{Q1Q2}
  Let  $\m$ be a measure in $\AF^d,$ $\Nb-2<d\le \Nb,$ and let the weight functions  $G_1,G_2$ satisfy the conditions of Lemma \ref{lemQ}, correspondingly, with $l=l_1$ and $l=l_2$. Set $\y_\io=\frac{2d}{d-\Nb+4l_{\io}},$ $\theta^{-1}=\y_1^{-1}+\y_2^{-1}=\frac{d-\Nb+2(l_1+l_2)}{d}.$ Then 
  \begin{equation}\label{Q1Q2estimate}
    \nb^{\sup}(\Zb,\theta)\le C \|G_1\|_{(\y_1),\m}^{\theta}\|G_2\|_{(\y_2),\m}^{\theta}.
  \end{equation}
\end{lem}
The crucial feature of Lemma \ref{Q1Q2} is that the order $\theta$ in the eigenvalue estimate of the product of operators is determined by the sum $l_1+l_2$, and not by $l_1,l_2$ separately.

Lemma \ref{Q1Q2} extends in the natural way to the product of several operators of the form $\Zb.$

By means of  transformations applied above, in Sect.4, the same estimate is justified for the measure $\m$ supported on the boundary of a domain in $\R^{\Nb}.$ 
\subsection{Spectrum of the difference of resolvents}\label{spectrum.perturb}
Now we combine the results of Sect.2 and Sect. \ref{subsect.trans} to obtain spectral  estimates for the difference  of resolvents of singular perturbed elliptic operators.

In conditions of Sect. \ref{subsect.trans}, we consider the difference of resolvents. Let $\Om\subseteq\R^\Nb$ be a connected open set. Let the  operator $\Ab$ in $\Omega$ be defined by the quadratic form $\ab$ in \eqref{form a concrete}  with some form-domain $\dF[\ab]\subset H^1(\Om)$. Further, let $\Si$
be a compact subset in $\Om,$ and let $\m$ be a measure supported in $\Si$, and let the  '$\mu-$ measurable weight function' $V$  satisfy \eqref{Condition 2}. Then the quadratic form $\vb[u]=\int V(X)|u(X)|^2\mu(dX)$ is bounded in $\dF[\ab]$ and, following the procedure described in Sect.2, the perturbed self-adjoint operator $\Ab_{\Vb}$ is defined.

According to the results of Sect.2, for the difference $\Rb_\Vb=\Ab^{-1}-\Ab_\Vb^{-1}$ of resolvents of these operators, formula \eqref{Main term} is valid,  with operator $\Tb\equiv\Tb_{\Vb}$ defined in \eqref{Tb0}. Having our results about eigenvalue estimates for the operators $\Tb$ in Theorem \ref{Thm.estimates} and Theorem \ref{Th.Est.Sb}, together with Lemma \ref{Lem.pert.eigenv}, we arrive at the following estimate.
\begin{thm}\label{th.main.estm.1} Under the assumption that \eqref{Condition 2} is satisfied, 
for the resolvent difference $\Rb_\Vb=\Ab^{-1}-\Ab_\Vb^{-1},$ the following spectral estimate is valid
\begin{equation}\label{Est.Diff.1}
  \nb^{\sup}(\Rb_\Vb, \theta)\le C \|V\|_{L_{(\theta),\m}}^{\theta}, \, \theta=\frac{d}{d-\Nb+4},
\end{equation}
with constant $C$ depending on the measure $\m,$ the operator $\Ab$ and the domain $\Om,$ but not depending on $V$. \end{thm}
Next we consider the resolvent difference for two singular perturbations, following Lemma \ref{prop.eigenv.diff}
\begin{thm}\label{th.main.estim.2}
Suppose that the  measure $\m$ and the weight functions  $V_1,V_2$ satisfy \eqref{Condition 2}. We construct  two perturbed operators $\Ab_{\Vb_1},\Ab_{\Vb_2} $. Then for the difference of resolvents of these perturbed operators, $\Rb^{(1,2)}=\Ab_{\Vb_1}^{-1}-\Ab_{\Vb_2}^{-1}$, the following spectral estimate holds
\begin{equation}\label{Est.Diff.2}
  \nb^{\sup}( \Rb^{(1,2)},\theta)\le C \|V_1-V_2\|_{L_{(\theta),\m}}^{\theta},  \theta=\frac{d}{d-\Nb+4},
\end{equation}
  again, with constant $C$ depending on the measure $\m,$ the operator $\Ab$ and the domain $\Om,$ but not depending on $V.$ 
  \end{thm} 
  Since $\vb_{1}[u]-\vb_2[u]=\int_{\Si}(V_1(X)-V_2(X))|u(X)|^2\mu(dX),$ the  proof follows immediately from Proposition \ref{prop.eigenv.diff} and Theorem \ref{Th.Est.Sb} applied to $V=V_1-V_2$.
  
  It follows from Theorem \ref{Thm.TR.lower} that the decay order of eigenvalues in  \eqref{Est.Diff.2} is sharp, at least for a sign-definite weight function $V_1-V_2.$  
\begin{rem}\label{Rem2measures}
A natural question here is  how Theorem \ref{th.main.estim.2} should be modified if the perturbations $\Vb_1,\Vb_2$ are defined by  two \emph{different} measures, $\m_1$, $\m_2$, with corresponding support $\Si_1,\Si_2$. First, if both measures belong to $\AF^d$, with the same $d,$ we can consider their sum $\m=\m_1+\m_2$ with support  $\Si=\Si_1\cup\Si_2$ which belongs to $\AF^d$ as well, extend the weight functions $V_j,\, j=1,2$ outside $\Sigma_{j}$ by zero  this reduces the problem to the case of one measure. 

Suppose now that $\m_1\in\AF^{d_1}, \m_2\in\AF^{d_2},$ with different dimension, $d_1<d_2$. Then the perturbation with larger $\theta$ makes a stronger contribution to the eigenvalue estimate. The result depends on the dimension of the space $\R^{\Nb}$. If $\Nb>4,$ we have $\theta_1=\frac{d_1}{d_1-\Nb+4}>\theta_2=\frac{d_2}{d_2-\Nb+4},$
therefore, in \eqref{R12proof}, the perturbation $V_2\m_2$ gives a weaker contribution to the spectral estimate. On the contrary, for $\Nb<4$, we have $\theta_1<\theta_2$, and therefore, it is the perturbation $V_2\m_2$ that gives a stronger contribution. The most delicate case is the dimension $\Nb=4,$ where one can see  that  $\theta_1=\theta_2=1.$ Here, the contributions of two perturbations are of the same order   in the eigenvalue estimate. In this case, unlike the case of $d_1=d_2,$ above, the contributions of positive and negative parts of $V_\io$ to eigenvalue estimates do not cancel.  These estimates have the form
\begin{equation*}
  \nb^{\sup}_{\pm}(\Db^{(1,2)},1)\le C ((\|(V_1)_{\pm}\|_{(1),\m_1}+\|(V_2)_{\mp}\|_{(1),\m_2})
\end{equation*}
(see details in \cite{GRInt}).
\end{rem}
\subsection{Estimates: some special cases}\label{subsect.special.estimates} We consider two examples.
\begin{exa}\label{Ex1}  Let $d=\Nb;$ this means that the measure $\mu$ is equivalent to  the Lebesgue measure $\Hc^\Nb.$ Let   the densities $V_1,V_2$ be compactly supported in $\Om.$  Then the exponent $\theta$ equals $\frac{\Nb}{4}$ and Theorem \ref{th.main.estim.2},  gives us the estimate
\begin{equation*}
  \nb^{\sup}( \Ab_{\Vb_1}^{-1}-\Ab_{\Vb_2}^{-1},\theta,)\le C \|V_1-V_2\|_{(\theta)}^{\theta}, \theta=\frac{\Nb}{4},
\end{equation*}
under the condition $V_1,V_2\in L_{(\theta)}.$ In particular, if $V_2=0,$ we have
\begin{equation*}
  \nb^{\sup}( \Ab^{-1}-\Ab_{\Vb_1}^{-1},\theta)\le C \|V_1\|_{(\theta)}^{\theta}, \theta=\frac{\Nb}{4}.
\end{equation*}
Essentially, these results are not new, see, in particular, \cite{Yafaev} and references therein. Note that in the literature on the scattering theory, it is the behavior of $V$ at infinity, $\Om=\R^\Nb$, that is of main interest, the topic we do not touch upon in \emph{this} paper.
\end{exa}
\begin{exa}\label{Exa2}Let $d=\Nb-1$. The condition $\m\in \AF^{\Nb-1}$ is satisfied, for example, for the case of the Hau{\ss}dorff measure $\m=\Hc^d$ on a Lipschitz hypersurface $\Si\subset\Om.$ Then the spectral problem for the operator $\Ab_{\Vb}$ is the transmission problem,
 consisting of finding the solution of the equation $\Lc u=\lambda u$ with jump condition on $\Si$:
$[\partial_{\n(\Lc)} u]_{|_{\Si}}(X)-V(X)u(X)=0, \, X\in\Si,$
 where brackets denote the jump in the conormal derivative of $u$ when crossing $\Si$, while the direction of the differentiation coincides with the direction of crossing.
 From a formally different point of view, this problem was considered in \cite{BeGuLo}, \cite{BeExL}, \cite{BLL.AHP}, where such perturbation was understood as a weighted $\de$-interaction supported on a  hypersurface $\Si$. The conditions in these papers were rather restrictive, in particular, both the hypersurface $\Si$ and the weight function were supposed to be smooth. The correct order for the spectral asymptotics was found there. 

Using the results of this paper, due to Theorem \ref{th.main.estim.2}, the exponent $\theta$ equals $\frac{N-1}{3}$, and
  our general eigenvalue estimates take the form
\begin{equation*}
 \nb^{\sup}(\Ab_{\Vb_1}^{-1}-\Ab_{\Vb_2}^{-1},\theta)\le C \|V_1-V_2\|_{(\theta)}^{\theta}, \theta=\frac{\Nb-1}{3}.
\end{equation*}
   We return to this setting in the next Section, where we justify the spectral asymptotics for $\Ab_{\Vb_1}^{-1}-\Ab_{\Vb_2}^{-1}$
\end{exa}

 \subsection{Singular Robin problem}\label{Subsect.Robin}We recall that the  classical Robin problem for the second order elliptic operator $\Lc$ in a domain $\Om\subset\R^\Nb$ consists in  finding the solution of the equation $\Lc  u=f$ with boundary conditions on $\G=\partial\Om$:
\begin{equation}\label{Robi.BC}
\partial_{\n(\Lc)}(X) u(X)-V(X)u(X)=0, \, X\in\G,
\end{equation}
where $\partial_{\n(\Lc)}(X)$ is the conormal boundary derivative associated with the operator $\Lc$. As mentioned in the Introduction, under certain regularity conditions, the singular numbers asymptotics  for the difference of resolvents of two operators of this form, was found in \cite{GrRobin}; we return to this issue later on. 

 In the present subsection  we consider a more general spectral problem, with the density $V$ in \eqref{Robi.BC} replaced by a measure $\pmb{\pi}=V\mu$ on the boundary. For the statement of the problem and initial estimates, this boundary  is supposed to be of the class $C^{0,1}$. Namely, we consider the spectral problem
\begin{equation*}
  \Ab_{\G,\Vb}u=\la u,
\end{equation*}
 with $\Ab_{\G,\Vb}$  being the  operator $\Lc$ in $\Om$ with, formally,  Robin type boundary conditions on $\partial \Om$ 
 \begin{equation}\label{Trans.bound}
   [\partial_{\n(\Lc)} u]_{|\G}(X)-V(X) u(X)\mu=0,
 \end{equation}
where $\mu$ is a measure on $\G$, possibly, singular, and $V$ is a real $\mu$-measurable function on $\G.$ This operator  is defined variationally, by means of a quadratic form.
Namely, generalizing \cite{BLL}, we associate with the Robin problem   \eqref{Trans.bound} the quadratic form
\begin{equation}\label{form.robin1}
  \ab_{\vb}[u]=\ab[u]+\vb[u], \, u\in H^1(\Om),
\end{equation}
where the perturbing quadratic form is 
\begin{equation}\label{form.robin2}
\vb[u]=\int_\G V(X)|\g u(X)|^2\m(dX);
\end{equation}
where $\g$, as before, is the restriction operator  of functions in $H^1(\Om)$ to the boundary.  Since the quadratic form \eqref{form.robin1} is defined on the space $H^1(\Om)$, the Robin problem, thus stated,  is a quadratic form perturbation of the Neumann operator $\Ab=\Ab^{\Nc}$.

\begin{proposition}\label{T.Robin} Suppose that the boundary $\G$ is Lipschitz and the measure $\m$ on $\G$ belongs to $\AF^d,$ $\Nb-2<d\le \Nb-1$. Let $V\in L_{(\vt),\m},$ $\vt=\frac{d}{d-\Nb+2}$.
Then  the   quadratic form \eqref{form.robin2} is bounded in $H^1(\Om)$,  the quadratic form \eqref{form.robin1} is semibounded. The operator $\Tb$ defined in \eqref{Tb0}, where, as previously, $\Fb$ is the operator of multiplication by the function $F=|V|^{\frac12}$ and $\Ub $ is the operator of multiplication by  $U=\sgn V$,  is compact and for its spectrum the estimates hold:
\begin{equation}\label{est.eigenv.robin}
  \nb_{\pm}^{\sup}(\Tb,\vt)\le C\|V_{\pm}\|_{L_{(\vt), \m}}^{\vt}, \,  \nb^{\sup}(\Tb,\vt)\le C\|V\|_{L_{(\vt), \m}}^{\vt}
\end{equation}
\end{proposition}
\begin{proof} Two first statements follow directly from    Lemma \ref{Lem.boundd.boundary.2}, with use of the Stein extension operator $\Jb:H^1(\Om)\to H^1(\R^{\Nb})$.  The same extension operator takes care of carrying over eigenvalue estimates in the whole space to the ones in the domain $\Om,$ using Lemma \ref{BSLem.extension}.
\end{proof}

Next, we  obtain spectral estimates for the operator
\begin{equation}\label{S.Robin}
 \Wb=\Ab^{-\frac12}\Tb\Ab^{-\frac12}=(\Fb\g\Ab^{-1})^*\Ub (\Fb\g\Ab^{-1})
\end{equation}
 in $L_2(\Om).$ Here we will need some more regularity of the boundary of $\Om.$
 
\begin{lem}\label{W.Robin} Let the boundary $\G$ belong to $C^{1,1}$ and the measure $\m$ and the weight $V$ on $\Om$ satisfy \eqref{Condition 2}. Then the singular numbers of $\Wb$ satisfy  estimate
\begin{equation*}
  \nb^{\sup}(\Wb,\theta)\le C \|V\|_{(\theta),\m}^{\theta}, \, \theta = \frac{d}{d-2N+4},
\end{equation*}
and similar estimates are valid for positive and negative eigenvalues of $\Wb$  separately.
\end{lem}  
\begin{proof} The operator $\Wb$ is unitarily equivalent to the operator defined by the quadratic form $\vb[u]$ in $\Dc(\Ab)$. Due to the $C^{1,1}$ regularity of the boundary, by Lemma \ref{Prop.H 2.2}, the domain $\Dc(\Ab)$ of the Neumann operator $\Ab$ is contained in $H^2$ in a neighborhood of  the boundary.  We apply, first, the restriction to this  neighborhood and then the $H^2$-extension to the whole $\R^\Nb,$ both times using Lemma \ref{BSLem.extension}.  Therefore   the task of estimating eigenvalues of the operator $\Wb$ is reduced  to estimating  the singular numbers of the operator defined by the quadratic form   \eqref{form.robin2} in the space $H^2(\R^\Nb),$  and the required estimate is provided by  Theorem \ref{th.main.estim.2} with $l=2$.
\end{proof}

Having established the required estimates for the operators $\Tb$ and $\Wb,$ we can now apply the abstract results of Sect. 2.

\begin{thm}\label{Th.5.8}Let the boundary $\G=\partial \Om$, the measure $\m$ and the weight functions $V_1$, $V_2$ satisfy the conditions of Lemma \ref{T.Robin} and Lemma \ref{W.Robin}.  Then for the generalized Robin operators,  the following estimates hold with $\theta=\frac{d}{d-\Nb+4}$:
 \begin{description}
   \item[(1)] For the operator $\Rb_{\Vb_1}=\Ab^{-1}-\Ab_{\Vb_1}^{-1}$, 
   \begin{equation*}
     \nb_{\pm}^{\sup}(\Rb_{\Vb_1},\theta)\le C\|V_{1,\pm}\|_{{(\theta)},\m}^{\theta};
   \end{equation*}
   \item[(2)] For the operator $\Rb^{(1,2)}=\Ab_{\Vb_1}^{-1}-\Ab_{\Vb_2}^{-1}$
   \begin{equation*}
     \nb_{\pm}^{\sup}(\Rb^{(1,2)},\theta)\le C\|(V_1-V_2)_{\pm}\|_{{(\theta)},\m}^{\theta}.
   \end{equation*}
 \end{description}
\end{thm}  
\begin{rem} The same discussion as in Remark \ref{Rem2measures} clarifies the situation with two different measures for two generalized Robin problems. 
\end{rem}
For the case when the measure $\m$  coincides with the surface measure on the boundary, $d=\Nb-1$, we obtain the eigenvalue estimate for the classical Robin problem in terms of integral characteristics of the density, improving the results of \cite{Birman.62}, \cite{BeGuLo}, \cite{GrRobin}, \cite{Malamud}. 
\begin{cor}\label{Cor.Robin} Let $\m$ be the surface measure on $\G=\partial{\Om},$ $d=\Nb-1,$ 
and the conditions of Lemma \ref{T.Robin} and \ref{W.Robin} be satisfied. Then $\theta=\frac{d}{d-\Nb+4}=\frac{\Nb-1}{3}$ and
\begin{equation}\label{est.Rob.Class}
  \nb_{\pm}^{\sup}(\Ab_{\Vb_1}-\Ab_{\Vb_{2}}, \frac{\Nb-1}{3}) \le C\|(V_1-V_2)_{\pm}\|_{L_{(\theta)},\m}^{\theta}, \, \theta=\frac{d}{d-\Nb+4}=\frac{\Nb-1}{3}. 
\end{equation}
 \end{cor}

\section{Eigenvalue asymptotics for resolvent difference.}\label{Sect.Asymptotic}

 For the case when $\Si$ is a  Lipschitz hypersurface in $\R^\Nb$, this means  of dimension $d=\Nb-1,$ and $\m$ is the surface  measure on $\Si$, thus, belonging to  $\AF^{\Nb-1}$, we can obtain the asymptotics of eigenvalues  of the operator $\Rb^{(1,2)}=\Ab^{-1}_{\Vb_1}-\Ab^{-1}_{\Vb_2}$. Moreover, when $\m$ is the Hau{\ss}dorff measure $\Hc^{\Nb-1}$ (and we can restrict ourselves to this case), such asymptotics can be proved for certain more general sets $\Si$.
  
  In  our initial motivation, namely,    for the difference of Robin resolvents with different densities, we can prove Weyl-type  asymptotic formulas for  a measure  $\mu\in \AF^{\Nb-1}$, however under the condition that  the boundary is smooth. 
This latter condition is explained by the fact that we use a perturbation approach based upon Lemma \ref{BSlemma}. This means that we approximate the 'singular' problem under consideration by more regular ones, where the eigenvalue asymptotics is already known. When we consider the perturbation contained strictly  inside our domain, we can use as approximating, the results on the asymptotics  for Lipschitz surfaces, obtained in \cite{RT1}. On the other hand, for the Robin problem, the only approximating results available are the ones in \cite{GrRobin}, where the boundary is smooth. It is unclear at the moment, how far one can relax these smoothness conditions, using either the pseudodifferential approach in  \cite{Grubb.nonsmooth} or the approximation approach developed in \cite{GRLip}.
 \subsection{$\de$-potentials}\label{sub.delta-potentials}
We consider first a more simple case, namely, when the measure $\m$ is in $\AF^{d} $, $d=\Nb-1$, on some compact set $\Si$ inside $\Om$;  equivalently,  $\m$ is the Hau{\ss}dorff measure $\Hc^{\Nb-1}$ on $\Si$. In other words, we study the difference of resolvents of two  Schr{\"o}dinger type operators
\begin{equation*}
  \Rb^{(1,2)}=\Ab_{\Vb_1}^{-1}-\Ab_{\Vb_2}^{-1},
\end{equation*}
in $L_2(\Om)$, with $\Hc^{d}$-measurable functions $V_1,V_2$ on $\Si$. When the set $\Si$ is a hypersurface in $\Om,$ say Lipschitz one, this setting corresponds to considering $\delta$-potentials on $\Si,$ with weights $V_1,V_2.$

 The corresponding  expression is formally the same as in \cite{BeGuLo}, \cite{BLL.AHP}. We are pretty sure that for sufficiently regular $\Si$ and piecewise smooth functions $V_1,V_2,$ it might be proved following the approach in \cite{GrRobin}, \cite{Grubb.sing}, \cite{AG} with use of the calculus of nonsmooth singular Green operators. We, however,  aim for obtaining the eigenvalue asymptotics under  considerably less restrictive conditions. We follow essentially the pattern elaborated in \cite{RTsing}, \cite{GRInt}, \cite{RSh}; therefore, we skip  technical  details and  restrict ourselves  to explaining  main steps only.

For a point $X$ on a Lipschitz hypersurface $\Si,$ the density $\om(X)$ is determined by the principal symbol $\aF(X;\x)=\sum_{j,j'}a_{j,j'}(X)\x_j\x_{j'}$ of the operator $\Lc$ at this point. This density is defined  $\m$-almost everywhere on $\Si,$ namely, at those points $X\in\Si$ where there exists the tangent hyperplane $\mathrm{T}_X\Si$, with the tangent unit sphere $\mathrm{S}_X\Si$,  and the corresponding normal vector $\nub(X):$
\begin{gather}\label{Symbol}
  \om(X)=\frac{1}{d(2\pi)^d}\int_{S_X\Si}\rb(X,\x')^{\theta}d\sigma(\x'), \, \theta=\frac{\Nb-1}{3};\\\nonumber
  \rb(X,\x')=(2\pi)^{-1}\int_{-\infty}^{\infty}\aF(X;\x',\y\nub(X))^{-2}d\y,
\end{gather}
see \cite{RTsing},
where $\sigma$ is the natural measure on the unit cotangent  sphere in $\R^{d-1}$,
see Theorem 6.2 in \cite{RT1} (for $\Nb\ne 4$) and Theorem 2.4 in \cite{RSh} (for $\Nb=4$).
  
\begin{thm}\label{Thm.Asympt}Let $\Si\subset \Om$ be a Lipschitz hypersurface in $\Om$, and $V_1,V_2$ be functions  in $L_{(\vt),\m}, $  $\vt={\Nb-1}.$ Then for the difference of resolvents $\Rb^{(1,2)}=\Ab_{\Vb_1}^{-1}-\Ab_{\Vb_2}^{-1}$ the asymptotic  formulas hold:
\begin{equation*}
  \nb_{\pm}(\Rb^{(1,2)}, \theta)=(2\pi)^{-d}\int_{\Si}\om(X)(V_2(X)-V_1(X))_{\pm}^{\theta} \Hc^{d}(dX),\, \theta=\frac{\Nb-1}{3},
\end{equation*}
where  $\om(X),$ $X\in\Si$  is the density  determined by the coefficients of the operator $\Lc$ at $X,$ see \eqref{Symbol}.  
\end{thm}
\begin{proof} As explained above, the asymptotics of eigenvalues of $\Rb^{(1,2)}$ is the same as the one for the compact operator
 \begin{equation*}
\Wb=\Wb^{(1,2)}=(|V_1-V_2|^{\frac12}\g\Ab^{-1})^*\Ub(|V_1-V_2|^{\frac12}\g\Ab^{-1}),
\end{equation*}
$\Ub=\sgn(V_1-V_2)$; such operators were considered systematically in \cite{RSh}, \cite{RTsing}, \cite{GRInt}.
In fact, the statement in the theorem above is a particular case of Theorem 2.4 in \cite{RSh} and Theorem 6.6 in \cite{GRInt} for the critical dimension $\Nb=4$, and Theorem 6.4 in \cite{RTsing} for the non-critical  dimension, $\Nb\ne 4$, granted that the estimate in Theorem \ref{Th.5.8} is already established.

For the benefit of the Readers, we explain here the structure of the proof, addressing to the above sources  for technical details.

We will use Lemma \ref{BSlemma} systematically. Having obtained spectral estimates, as in Theorem \ref{Th.5.8}, involving \emph{integral norm} of the weight functions, we approximate \emph{in this integral norm over $\Si$} the given  weight functions by more and more regular or more convenient ones, arriving at last  to the situation where the eigenvalue asymptotics is already known. We denote $V=V_1-V_2$ and perform the following steps.

 \begin{enumerate}
 \item The function $V(X)$ is approximated (in the above integral norm) by the restriction to $\Si$ of a \emph{smooth} function $V^{(1)}(X), \, X\in \R^{\Nb}$ belonging to $C_0^\infty$ in a neighborhood of $\Si.$
 \item The function $V^{(1)}$ is approximated (in the above integral norm, again) by a smooth function $V^{(2)}$ such that  the set where it is  positive  and the set where it is negative are separated, $\dist(\supp(V^{(2)}_+), \supp(V^{(2)}_-))>0.$ After this, the property 'sign separation' established in the papers cited above, is used: for \emph{such} weight function $V^{(2)}$, the leading order asymptotic characteristics of positive eigenvalues of $\Wb$ do not depend on the negative part, $V^{(2)}_-$, and vice versa. Therefore, we can restrict our considerations to a smooth non-negative  function $V^{(2)}\ge0$.   
        \item Consider now the  spectrum of the (nonnegative) operator 
        \begin{equation*}
        \Wb^{(2)}=[(V^{(2)})^{\frac12}\g\Ab^{-1}]^*[(V^{(2)})^{\frac12}\g\Ab^{-1}].  
\end{equation*}
Here we  may commute the (smooth!) function $(V^{(2)})^{\frac12}$ and the restriction operator $\g,$
            
            \begin{equation}\label{M}
              \Wb^{(2)}=\left[\g(V^{(2)})^{\frac12}\Ab^{-1}\right]^*\left[\g(V^{(2)})^{\frac12}\Ab^{-1}\right].
            \end{equation}
The nonzero (in fact, positive) eigenvalues of the operator \eqref{M} in $L_2(\Om)$ coincide with nonzero eigenvalues of the operator
\begin{equation}\label{M'}
  \Wb^{(3)}=[\g(V^{(2)})^{\frac12}\Ab^{-1}][\g(V^{(2)})^{\frac12}\Ab^{-1}]^*
\end{equation}
acting in $L_2(\Si,\m)$.
\item The operator $(V^{(2)})^{\frac12}\Ab^{-1}$ is, in fact, up to a smoother remainder, an integral operator in $\Om$ with kernel $(V^{(2)})^{\frac12}(X)G_{-2}(X,Y),$ where $G_{-2}$ is the fundamental solution of the operator $\Ab$. The leading term of this fundamental solution is the Fourier transform of the inverse symbol $\aF(X,\x)^{-1}$  in $\R^\Nb$ and, therefore, is the integral kernel with  singularity at the  diagonal. On this stage, we pass from the pseudodifferential setting, where certain regularity of symbols and surfaces is required, to the setting involving integral operators, where considerably less regularity is needed.   After the restriction to $\Si$, this means, after the application of $\g$, it follows that  $\g\circ (V^{(2)})^{\frac12}\Ab^{-1}$ is the integral operator with the same kernel, but acting from $L_2(\Om)$ to $L_2(\Si,\m)$. The adjoint, $[\g\circ  (V^{(2)})^{\frac12}\Ab^{-1}]^*$, is, therefore, the integral operator with the kernel $(V^{(2)})^{\frac12}(Y)G_{-2}(X,Y)$ acting from $L_2(\Si,\m)$ to $L_2(\Om)$. This kernel has power type or (for $\Nb=2$) logarithmic singularity at the diagonal $X=Y$.  Finally, for the composition in \eqref{M'}, we see that $\Wb^{(3)}$ is the integral operator on $\G,$ with leading singularity in the kernel:
     \begin{equation}\label{M1M}
      \Wc(X,Y)=(V^{(2)}(X))^{\frac12}\int_{\Om} G_{-2}(X,Z)G_{-2}(Y,Z)dZ(V^{(2)}(Y))^{\frac12}.
    \end{equation}
    \item The composition of two copies of the  fundamental solution of $\Lc$ in \eqref{M1M} produces, again,  up to some smoother remainder terms, the fundamental solution of the \emph{fourth } order elliptic operator $\Lc^2$. The leading term in this  fundamental solution 
has the  power, or $\log$-power for $\Nb=2$, or logarithmic  for  $\Nb=4$, singularity at the diagonal  $X=Y.$
\item We can apply now the results of the paper \cite{RT1}, where for integral operators on a Lipschitz surface, with power or $\log$-power singularity at the diagonal, the eigenvalue asymptotic formulas were obtained. This takes care of the statement of Theorem \ref{Thm.Asympt} for a single Lipschitz surface.
    
            \end{enumerate}          \end{proof}
     In the papers \cite{GRInt}, \cite{RTsing}, asymptotic formulas for operators of the type $\Wb$ were established for the case when the Lipschitz surface   
     $\Si$ is replaced by a more general, uniformly rectifiable, set.
    \begin{defin}Let $d=\Nb-1.$ The set $\Si\subset\R^\Nb$ is called \emph{uniformly rectifiable} (of dimension $d$) if it is, up to a set of zero Hau{\ss}dorff measure  $\Hc^d$,  the union of a countable collection of Lipschitz surfaces. 
\end{defin}
A detailed description of uniformly rectifiable sets and an extensive discussion of their properties can be found in the books \cite{Mat1}, \cite{David},  \cite{falconer} and other sources on the geometric measure theory.

In the papers \cite{GRInt}, \cite{RTsing}, the eigenvalue asymptotic formulas were obtained for the case when the operator $\Lc$ is the Laplacian, where the density  $\om(X)$ in \eqref{Symbol} can be calculated explicitly. For a general second order elliptic operator $\Lc,$ the expression for this density is quite unwieldy.

     In the proof, in order to pass to the set $\Si$ consisting of a \emph{finite} collection of Lipschitz surfaces, one more approximation is used, reducing the problem to disjoint surfaces; here the spectrum of the integral operator, again, up to a weaker term, is the union of the spectra operators on separate surfaces, which justifies the asymptotic formula for this case.
         Finally, for a uniformly rectifiable set, the final approximation is performed, reducing the problem to the previous case, using, for the last time, the estimate in Theorem \ref{Th.5.8}. The detailed procedure is described in \cite{GRInt}, \cite{RTsing}. 

\subsection{Regular Schr{\"o}dinger operator}\label{regular}  For completeness, we explain how our approach works when the measure $\m$ is the Lebesgue measure on $\R^{\Nb}$, and thus the potentials $V_\io\mu$, $\io=1,2$ can be considered as usual functions. In this case, $d=\Nb$ and the exponent $\theta$ equals $\frac{d}{4},$ while $\vt$ equals $\frac{d}{2}$.

In these conditions, by Theorems \ref{th.main.estm.1}, \ref{th.main.estim.2}, the estimate
\begin{equation*}
  \nb^{\sup}(\Wb,\theta)\le C \|V_1-V_2\|_{(\theta)}^{\theta}, \theta=\frac{\Nb}{4},
\end{equation*}
is valid, under the condition that $V_1,V_2\in L_{(\vt)}.$ Now, similarly to our reasoning above, 
we obtain the asymptotics
\begin{equation*}
  \nb_{\pm}(\Rb^{(1,2)}, \theta)= C_\Nb\int(V_1(X)-V_2(X))_{\pm}^{\frac{d}4}\det(\aF(X))^{-\frac14}dX.
\end{equation*} 
 For the Laplacian, such result seems to be folklore since early 1970-s, hinted for by Birman-Solomak, Simon, Rozenblum.

\subsection{The Robin problem} Here, we suppose that the boundary is $C^\infty$-smooth. This restriction is caused by the fact that the existing results on the eigenvalue asymptotics for the difference of Robin resolvents are established for smooth boundaries only. As soon as the asymptotic formulas are proved for less smooth boundary and for a smooth weight function, the asymptotic formula for nonsmooth weights  will follow automatically.
\begin{thm}\label{Thm.Robin} Let $\Om$ be a domain with smooth boundary and let $\Lc$ be a second order elliptic operator. We suppose that $\m$ is the Hau{\ss}dorff measure $\Hc^d$ on $\partial\Om$  and, finally, $V_1,V_2$ satisfy \eqref{Condition 2}. Then for the eigenvalues of the resolvent difference of Robin operators in $\Om$ with densities $V_1,V_2$,
the asymptotic formulas  hold, with $d=\Nb-1$, $\theta=\frac{d}{3}$:
\begin{equation}\label{Robin.As}
  \nb_{\pm}(\Rb^{(1,2)},\theta)=\frac1{d(2\pi^d)} \int_{\partial\Om}\om_R(X)(V_1(X)-V_2(X))_{\pm}^{\frac{d}3}dX, 
\end{equation}
where $\om_R(X)$ is given by the integrand in Theorem 3.5 in \cite{GrRobin}.
 \end{thm}
 The proof, up to the step 3, follows the one in Theorem \ref{Thm.Asympt}. The difference is that in the expression for the operator $\Wb$ in \eqref{S.Robin}, the unperturbed operator $\Ab$ is not an elliptic operator in $\R^\Nb$ but the Neumann operator in $\Om.$ Under the condition that $\partial\Om$ is of class $C^{1,1},$ the spectral estimate for the operator $\Wb$ in Theorem \ref{Th.5.8} has the same form as for the interior potential; it involves some integral norm of $V=V_1-V_2.$ Using this estimate, we can repeat the reasoning in Theorem  \ref{Thm.Asympt}, this means, separate positive and negative parts of the density $V$ and then perform the approximation by smooth positive densities, using Lemma \ref{BSlemma}.
 \subsection{The case of a fractional dimension $d$}\label{subsec.frac}
 A natural question arises, whether there are asymptotic eigenvalue formulas for the operator $\Rb^{(1,2)}$ with measure $\m, $ Ahlfors regular with a fractional dimension $d$. Our reference to the result by H.Triebel \cite{Tr} shows that, generally, the asymptotic \emph{order} $\theta=\frac{d}{d-\Nb+4}$ is sharp.   
 
 In the literature, there are a number of studies of the  spectrum of operators of the form $\Tb$ above; sometimes they are called Krein-Feller operators.  Main interest there is directed to the cases when the measure $\m$ and the density $V$ have some regular fractal structure; in such cases, the spectrum can be studied using arithmetic and algebraic relations. The leading examples of measures in $\AF^d$ with fractional $d$ (the construction is valid for integer $d$ as well, of course) are the ones obtained as self-similar measures, see Sect. 3.2.

 The results existing in the literature concern the spectrum of operators of the form $\Tb.$ The author failed while searching for the results for the operator $\Wb.$  
 
 For $\Tb,$ even in the one-dimensional case, the asymptotics for $n(\la)$ can be non-power-like. Namely, see, e.g.,  \cite{SolVer}, where, depending on the arithmetic structure of the self-similar singular measure, the class of measures was described, where the eigenvalue asymptotics is
  \begin{equation}\label{SolVerb}
  n(\Tb,\la)\sim \la^{-d} \F(\log\la)+o(\la^{-d}),
  \end{equation}
 with $d$ being  the Hau{\ss}dorff dimension and $\F$ being  a bounded periodic function. By passing to the direct product, one obtains similar examples with non-power asymptotics in higher dimensions. This effect was also studied in \cite{NaiSol}, \cite{Kigami} and many other papers in an arbitrary dimension.  Thus, at least for now, it seems to be non-realistic to search for eigenvalue asymptotics in a general case, without structural regularity. Nevertheless, our results in Sect. \ref{Sec.Sing.estim} demonstrate that the upper  eigenvalues estimates are order sharp.

\section{The difference of powers of resolvents}\label{Sect.powers} The reasoning leading to the eigenvalue estimates above, produces  spectral estimates for the difference of powers of resolvents. In order to simplify the presentation, we discuss Schr{\"o}dinger-like operators in  $\R^\Nb$ with singular potentials only. The case of Robin operators is treated similarly.

We recall the expressions for $\Ab_{\Vb}^{-1}$ in  \eqref{inverse},
\begin{equation}\label{inverse.bis}
  (\Ab_{\Vb})^{-1}=\Ab^{-\frac12}(1+\Tb)^{-1}\Ab^{-\frac12},\, \Tb=(\Fb\g\Ab^{-\frac12})^*\Ub(\Fb\g\Ab^{-\frac12}),
\end{equation}
thus,
\begin{equation}\label{Tb.bis}
  (1+\Tb)^{-1}=1-\Tb+ \Tb (1+\Tb)^{-1}\Tb\equiv 1-\Tb+\Tb'.
\end{equation}

We consider the operator

 \begin{equation}\label{AbN}
 (\Ab_{\Vb})^{-m}=\left[\Ab^{-\frac12}(1-\Tb+\Tb')\Ab^{-\frac12}\right]^{m},
\end{equation}
for an integer $m>1.$
The operator in \eqref{AbN} splits into the sum of terms of the following form:
\begin{description}
  \item[H1] $\Ab^{-m}$;
  \item[H2] $m$ terms containing one factor $\Wb=\Ab^{-\frac12}\Tb\Ab^{-\frac12}$ and $m-1$ factors $\Ab^{-1}$;
  \item[H3] terms containing two factors $\Wb $ and $m-2$ factors $\Ab^{-1}$
  \item[H4] the remaining terms, those that have at least one factor $\Ab^{-\frac12}\Tb'\Ab^{-\frac12}$ and those that have at least three factors $\Wb$.
\end{description}
In calculation of the resolvent difference $\Ab_{\Vb}^{-m}-\Ab^{-m}$, the term of the form \textbf{H1} cancels.
Consider some term of the form \textbf{H2}. Such term is obtained from \textbf{H1} by means of replacing one of $m$ factors $\Ab^{-1}$ by $\Wb$, so it   has the form 
\begin{equation}\label{T2.1}
 \Rb_{m_1,m_2} \Ab^{-m_1}\Ab^{-1/2}\Tb\Ab^{-1/2} \Ab^{-m_2}=\Qb_1^*\Qb_2,\, m_1+m_2=m-1,
\end{equation}
where $\Qb_1=UF\g\Ab^{-m_1-1},$ $\Qb_2=F\g\Ab^{-m_2-1}$, and as previously, $F=|V|^{\frac12}$, $U=\sgn(V).$

The singular numbers of such operator are estimated by means of  Lemma \ref{Q1Q2}, with $l_1=m_1+1,$ $l_2=m_2+1.$ This gives
\begin{gather}\label{nQ1Q2}
  n^{\sup}( \Rb_{m_1,m_2},\theta)\le C \\\nonumber \|F_1\|^{\theta}_{(\theta_1),\m}\|F_2\|^{\theta}_{(\theta_2),\m}, \\\nonumber \theta^{-1}=\frac{d-\Nb+2(m_1+m_2+2)}{d}=\frac{d-\Nb+2(m+1)}{d}.
\end{gather}
Note that all terms of this form produce the same order of eigenvalue estimates, with $\theta=\frac{d}{d-\Nb+2(m+1)}$. This will be the leading term in the spectral estimate for $\Ab_{\Vb}^{-1}-\Ab^{-1}.$

Next, we consider  terms of type \textbf{H3}; each such term is now obtained by replacing \emph{two} factors $\Ab^{-1}$ in $\Ab^{-m}$ by $\Wb$. So, such operator can be represented as
\begin{equation}\label{T3.term}
  \Rb_{m_1,m_2,m_3}=\Ab^{-m_1}\Ab^{-1/2}\Tb\Ab^{-1/2}\Ab^{-m_2}\Ab^{-1/2}\Tb\Ab^{-1/2}\Ab^{-m_3}, \, m_1+m_2+m_3=m-2.
\end{equation}
Such operator can be represented as the product of two operators of the form \textbf{H2},  
 \begin{equation}\label{T3Q1Q2} \Rb_{m_1,m_2,m_3}=\left[\Ab^{-m_1}\Ab^{-1/2}\Tb\Ab^{-1/2}\Ab^{-m_2'}\right]\left[\Ab^{-m_2''}\Ab^{-1/2}\Tb\Ab^{-1/2}\Ab^{-m_3}\right].
\end{equation}
with $m_2'+m_2''=m_2$. 
  We apply the previous calculation, using Lemma \ref{Q1Q2}, for each of factors in \eqref{T3Q1Q2} and then the Ky Fan inequality for the product of operators, to obtain   
  \begin{equation}\label{T3Q1Q2est}
    n(\la,\Rb_{m_1,m_2,m_3})=O(\la^{-\theta^{(3)}}), \, \theta^{(3)} =\frac{d}{d-\Nb+2(m+2)}<\theta=\frac{d}{d-\Nb+2(m+1)}.
  \end{equation}

  For terms of type \textbf{H4}, similar calculation shows that the more factors $\Ab^{-1/2}\Tb\Ab^{-1/2}$ are inserted in the term of the expansion of
$\Ab_{\Vb}^{-m},$ the faster is the singular numbers decay for this term. 

Finally, if some factor $\Ab^{-1}$ is replaced by $\Tb'=\Tb(1+\Tb)^{-1}\Tb$, the resulting operator $\Rb'$ can be represented as

\begin{gather*}
  \Rb'=\Ab^{-m_1-\frac12}\Tb'\Ab^{-m_2-\frac12}=\Qb_1^*\Qb_2(1+\Tb)^{-1}\Qb_2*^\Qb_3,\\\nonumber
  \Qb_1=F\g\Ab^{-m_1-1}, \Qb_2=F\g \Ab^{-\frac12}, \Qb_3=F\g\Ab^{-m_2-1}.
\end{gather*}
To estimate  the singular numbers of the latter operator we again apply Lemma \ref{Q1Q2} and then the Ky Fan inequality, which gives, again,
\begin{equation*}
  n(\la,\Rb')=O(\la^{-\theta^{3}}.
\end{equation*}
Terms of the type \textbf{H4} containing more additional compact factors $\Tb$ or  $\Tb'$, have, by Lemma \ref{lem.comp} a faster singular numbers decay than the main term \textbf{H2}.

The passage to estimates for the resolvent difference for two perturbations, $\Vb_1$ and $\Vb_2$ follows the pattern of Theorem \ref{th.main.estim.2}, with the same use of Lemma \ref{prop.eigenv.diff}.
As a result, we have the following estimate for eigenvalues of the difference of powers of resolvents.

\begin{thm}\label{thm.diff.powers}
  Let the measure $\m$ and the weight functions $V_1,V_2$ satisfy condition \eqref{Condition 2}.
  Then for the singular numbers of the operator $\Ab_{V_1}^{-m}-\Ab_{\Vb_2}^{-m}$ the estimate holds
  \begin{equation}\label{Final estimate for powers}
    \nb^{\sup}(\Ab_{V_1}^{-m}-\Ab_{\Vb_2}^{-m},\theta)\le C\|V_1-V_2\|_{(\theta),\m}^{\theta}, \, \theta=\frac{d}{d-\Nb+2(m+1)}.
  \end{equation}
\end{thm}
It follows, in particular, that for $m$ large enough, the difference of resolvents belongs to the trace class, which leads to standard consequences of the scattering theory.

\end{document}